\documentclass[12pt,twoside]{amsart}

\usepackage{amsmath,amsthm,amscd,amssymb,mathrsfs,graphicx,amsfonts,mathrsfs}
\usepackage{amsfonts}
\usepackage{amssymb,enumerate}
\usepackage{amsthm}
\usepackage[all]{xy}
\usepackage{hyperref}
\allowdisplaybreaks[4] \footskip=12pt
\renewcommand{\uppercasenonmath}[1]{}

\numberwithin{equation}{section} \theoremstyle{plain}
\newtheorem*{thm*}{Main Theorem}
\newtheorem{thm}{Theorem}[section]
\newtheorem{cor}[thm]{Corollary}
\newtheorem*{cor*}{Corollary}
\newtheorem{lem}[thm]{Lemma}
\newtheorem*{lem*}{Lemma}
\newtheorem{fact}[thm]{Fact}
\newtheorem*{fact*}{Fact}

\newtheorem*{nota*}{Notation}
\newtheorem{prop}[thm]{Proposition}
\newtheorem*{prop*}{Proposition}
\newtheorem{rem}[thm]{Remark}
\newtheorem*{rem*}{Remark}

\newtheorem*{observation*}{Observation}

\newtheorem*{exa*}{Example}
\newtheorem{df}[thm]{Definition}
\newtheorem*{df*}{Definition}

\newtheorem*{con*}{Construction}

\renewcommand{\geq}{\geqslant}
\renewcommand{\leq}{\leqslant}

\begin{document}
\begin{center}
{\large  \bf  Induced complete hereditary cotorsion pairs in $\mathrm{D}(R)$ with respect to Cartan-Eilenberg exact sequences}

\vspace{0.5cm}  Xiaoyan Yang\\
School of Science, Zhejiang University of Science and Technology, Hangzhou 310023\\
E-mail: yangxy@zust.edu.cn
\end{center}

\bigskip
\centerline { \bf  Abstract}
\leftskip10truemm \rightskip10truemm \noindent Given a  complete hereditary cotorsion pair $(\textsf{A},\textsf{B})$ in $\mathrm{Mod}R$, we construct a complete hereditary cotorsion pair in the derived category $\mathrm{D}(R)$ of unbounded complexes with respect to the proper class $\xi$ of cohomologically ghost triangles induced by the Cartan-Eilenberg exact sequences. More specifically, we prove that, each of the classes of projectively coresolved $\xi$-$\mathcal{G}$flat complexes $\mathcal{PGF}(\xi)$, $\xi$-$\mathcal{G}$flat complexes $\mathcal{GF}(\xi)$, $\xi$-$\mathcal{G}$injective complexes $\mathcal{GI}(\xi)$, $\xi$-$\mathcal{G}$projective complexes $\mathcal{GP}(\xi)$ (the last when $R$ is virtually Gorenstein), forms one half of a complete hereditary cotorsion pair in $\mathrm{D}(R)$ with respect to $\xi$. Moreover, various homological dimensions offer additional way to obtain such cotorsion pairs in $\mathrm{D}(R)$ with respect to $\xi$.
\leftskip10truemm \rightskip10truemm \noindent\\[2mm]

\leftskip0truemm \rightskip0truemm
\bigskip
\section* { \bf Introduction}

Since their introduction by Salce \cite{S} in 1979,  cotorsion pairs have evolved from a specialized concept in module theory into a cornerstone of homological algebra, representation theory and algebraic geometry.
A complete cotorsion pair induces a model structure on the abelian category, providing the foundational framework for homotopy theory and homological constructions. A key result in this context, due to Trlifaj and $\mathrm{\check{S}}$aroch and relying on the Eklof-Trlifaj lemma \cite{ET}: in a Grothendieck category with enough projectives, the cotorsion pair $({}^\perp(\mathcal{S}^\perp), \mathcal{S}^\perp)$ cogenerated by any set $\mathcal{S}$ of objects is complete. This result provides a powerful method for constructing complete cotorsion pairs. For example, the flat-cotorsion pair $(\mathsf{F},\mathsf{C})$ was shown to be complete by Bican, El Bashir and Enochs \cite{BBE},  setting a long-standing conjecture.

Connections between model structures and cotorsion pairs were first established by Hovey \cite{H02} and Beligiannis-Reiten
\cite{BR}. There exists an one-one correspondence between abelian model structures and Hovey triples
(see \cite[Theorem 2.2]{H02}).
  A Hovey triple gives rise to two compatible complete cotorsion pairs; and if the two compatible complete cotorsion pairs are hereditary, then the Hovey triple is uniquely determined by the two cotorsion pairs (see \cite[Theorem 1.1]{G15}). Thus, complete cotorsion pairs provide a foundational framework for studying algebraic structures.
  More specifically, any complete hereditary cotorsion pair in $\mathrm{Mod}R$ naturally induces two such pairs in $\mathrm{C}(R)$, to which Hovey's theorem may be applied (see  \cite[Theorem 3.5]{YL11} or \cite[Theorem 2.4]{YD}). Hovey's correspondence has also been extended by Gillespie
\cite{G11} and $\mathrm{\check{S}}\mathrm{\check{t}}\textrm{ov}\acute{{\i}}\mathrm{\check{c}}\textrm{ek}$ \cite{St} to a one-one correspondence between exact model
 structures and  Hovey triples in weakly idempotent complete exact categories. Furthermore, Yang \cite{Y} generalized this correspondence to triangulated categories equipped with a proper class $\xi$ of triangles, establishing a $\xi$-triangulated analogue of Hovey's correspondence.

A Hovey triple consists two complete cotorsion pairs. Beligiannis and Reiten \cite{BR} constructed model structures on
 abelian categories from only one complete hereditary  cotorsion pair. These are generally weaker than
  abelian model structures. Cui-Lu-Zhang \cite[Theorem 1.1]{CLZ} extended this results to weakly idempotent
 complete exact categories. More recently, Hu-Zhang-Zhou further generalized the work of Beligiannis-Reiten and Cui-Lu-Zhang, extending this constructioin
to  weakly idempotent complete extriangulated categories (see \cite[Theorem 1.1]{HZZ}. Extriangulated category introduced in \cite{NP}, provide a
 simultaneous generalization of exact categories and triangulated categories.

 Auslander and Bridger \cite{AB} introduced and studied modules of G-dimension 0 for finitely generated modules over two-sided noetherian rings. To extend this theory to arbitrary modules, Enochs-Jenda-Torrecillas
\cite{EJ95,EJT} introduced Gorenstein projective, Gorenstein injective  and Gorenstein flat modules, along with their associated dimensions (also see \cite{EJ,H}). An
important  recent advance is the introduction of projectively coresolved Gorenstein flat
modules over any  ring $R$, by $\mathrm{\check{S}}$aroch and $\mathrm{\check{S}}\mathrm{\check{t}}\textrm{ov}\acute{{\i}}\mathrm{\check{c}}\textrm{ek}$ \cite{SS}.

Significantly, over certain rings, the classes of Gorenstein projective, Gorenstein injective, Gorenstein flat, projectively coresolved Gorenstein flat  modules, along with their associated dimensions, each form complete hereditary cotorsion pairs in $\mathrm{Mod}R$ (see \cite{GLZ,M,WE,SS}). This structural insight provides a powerful framework for studying Gorenstein rings and Calabi-Yau categories. Moreover, the theory has been successfully extended to categories of quasi-coherent sheaves on schemes, offering new tools for investigating vector bundles and cohomology on singular varieties.

Inspired by the work of Enochs and Jenda in a triangulated setting, Asadollahi and Salarian \cite{AS} introduced $\xi$-$\mathcal{G}$projective and $\xi$-$\mathcal{G}$injective  objects, along with their relative dimensions, with respect to a proper class $\xi$ of triangles.  However, finding a proper class $\xi$ of triangles with enough $\xi$-projectives in a general triangulated category $\mathcal{T}$ remains nontrivial. Let $R$ be a ring with identity. The derived category $\mathrm{D}(R)$ obtained by
formally inverting the quasi-isomorphisms in the category
$\mathrm{C}(R)$ of unbounded complexes.
 An important example of such a proper class $\xi$ in $\mathcal{T}:=\mathrm{D}(R)$ is provided by Beligiannis \cite{B}, based on Cartan-Eilenberg (CE) exact sequences of complexes, described as \emph{cohomologically ghost triangles} in \cite{O}. By \cite[Theorem 3.2]{HZZ0}, $(\mathrm{D}(R),[1],\xi)$ is an extriangulated category, although it is not triangulated.

 In the category $\mathrm{C}(R)$ with Cartan-Eilenberg (CE) exact structure, CE projective and CE injective complexes serve as relative projectives and injectives. Enochs \cite{E} studied the CE projective, CE injective  and CE flat complexes and showed that they induce complete hereditary cotorsion pairs relative to a subfunctor $\overline{\mathrm{Ext}}$ of Ext. Moreover, any hereditary cotorsion pair in $\mathrm{Mod}R$  that is cogenerated by a set can be lifted to a hereditary cotorsion pair in $\mathrm{C}(R)$ that is cogenerated by a set with respect to $\overline{\mathrm{Ext}}$, see  \cite[Theorem 9.4]{E}.

 Let $\textsf{X}$ be a class of $R$-modules. Denote $$\mathcal{X}_\textsf{P}:=\{X\in\mathrm{D}(R)\hspace{0.01cm}|\hspace{0.01cm}\mathrm{coker}(P^{s-1}\rightarrow P^s)\in\textsf{X}\ \textrm{for\ some\ semi-projective resolution}\ P\stackrel{\simeq}\rightarrow X\ \textrm{and\ all}\ s\in\mathbb{Z} \},$$ $$\mathcal{X}_\textsf{I}:=\{Y\in\mathrm{D}(R)\hspace{0.01cm}|\hspace{0.01cm}\mathrm{ker}(I^{s}\rightarrow I^{s+1})\in\textsf{X}\ \textrm{for\ some\ semi-injective resolution}\ Y\stackrel{\simeq}\rightarrow I\ \textrm{and\ all}\ s\in\mathbb{Z}\}.$$
The present paper establishes the following main theorem.

\vspace{2mm} \noindent{\bf Theorem A}\label{Th1.4} {\rm(Theorem \ref{lem2.5}).} {\it{Let $R$ be a ring with identity. Then $(\textsf{A},\textsf{B})$ is a complete hereditary cotorsion pair in $\mathrm{Mod}R$ if and only if
$(\mathcal{A}_\textsf{P},\mathcal{B}_\textsf{I})$ is a complete hereditary cotorsion pair in $\mathrm{D}(R)$ with respect to the proper class $\xi$ of cohomologically ghost triangles.}}
\vspace{2mm}

 The definition of $\mathrm{D}(R)$ does not provide the necessary tools to prove non-trivial properties.
One significant advantage  of model categories is that it offer a powerful framework for studying $\mathrm{D}(R)$ in many contents. As applications of the theorem A, we examine the following classes of objects in $\mathrm{D}(R)$ relative to the proper class $\xi$  of cohomologically ghost triangles:
\begin{center}$\begin{aligned}&\mathcal{GP}(\xi):=\{\xi\textrm{-}\mathcal{G}\textrm{projective\ complexes}\},\\
&\mathcal{GI}(\xi):=\{\xi\textrm{-}\mathcal{G}\textrm{injective\ complexes}\},\\
&\mathcal{GF}(\xi):=\{\xi\textrm{-}\mathcal{G}\textrm{flat\ complexes}\},\\
&\mathcal{PGF}(\xi):=\{\textrm{projectively\ coresolved}\ \xi\textrm{-}\mathcal{G}\textrm{flat\ complexes}\},\\
&\mathcal{P}(\xi)_n:=\{\textrm{complexes\ of}\ \xi\textrm{-}\textrm{projective\ dimension}\ \leq n\},\\
&\mathcal{F}(\xi)_n:=\{\textrm{complexes\ of}\ \xi\textrm{-}\textrm{flat\ dimension}\ \leq n\},\\
&\mathcal{I}(\xi)_n:=\{\textrm{complexes\ of}\ \xi\textrm{-}\textrm{injective\ dimension}\ \leq n\},\\
&\mathcal{GP}(\xi)_n:=\{\textrm{complexes\ of}\ \xi\textrm{-}\mathcal{G}\textrm{projective\ dimension}\ \leq n\},\\
&\mathcal{GI}(\xi)_n:=\{\textrm{complexes\ of}\ \xi\textrm{-}\mathcal{G}\textrm{injective\ dimension}\ \leq n\},\\
&\mathcal{GF}(\xi)_n:=\{\textrm{complexes\ of}\ \xi\textrm{-}\mathcal{G}\textrm{flat\ dimension}\ \leq n\},\\
&\mathcal{PGF}(\xi)_n:=\{\textrm{complexes\ of\ projectively\ coresolved}\ \xi\textrm{-}\mathcal{G}\textrm{flat\ dimension}\ \leq n\}.\end{aligned}$\end{center}
 We show that each of these classes constitutes one half of a complete hereditary cotorsion pair in $\mathrm{D}(R)$ with respect to $\xi$. Applying
Hu-Zhang-Zhou's theorem then yields numerous model structures in the extriangulated category $(\mathrm{D}(R),[1],\xi)$.

\vspace{2mm} \noindent{\bf Theorem B}\label{Th1.4} {\rm(Corollaries \ref{lem1.5} and \ref{lem3.2}).} {\it{Let $R$ be a ring with identity and $n\geq0$.

$(1)$ $(\mathcal{P}(\xi)_n,\mathcal{P}(\xi)^{\bot_\xi}_n)$, $(\mathcal{F}(\xi)_n,\mathcal{F}(\xi)^{\bot_\xi}_n)$ and $({^{\bot_\xi}}\mathcal{I}(\xi)_n,\mathcal{I}(\xi)_n)$ are complete hereditary cotorsion pairs in $\mathrm{D}(R)$ with respect to $\xi$.

$(2)$ $(\mathcal{PGF}(\xi),\mathcal{PGF}(\xi)^{\bot_\xi})$ and $(\mathcal{PGF}(\xi)_n,\mathcal{P}(\xi)^{\bot_\xi}_n\cap\mathcal{PGF}(\xi)^{\bot_\xi})$ are complete hereditary cotorsion pairs in $\mathrm{D}(R)$ with respect to $\xi$.

$(3)$ $(\mathcal{GF}(\xi),\mathcal{GF}(\xi)^{\bot_\xi})$ and $(\mathcal{GF}(\xi)_n,\mathcal{F}(\xi)^{\bot_\xi}_n\cap\mathcal{PGF}(\xi)^{\bot_\xi})$ are complete hereditary cotorsion pairs in $\mathrm{D}(R)$ with respect to $\xi$.

$(4)$ $({^{\bot_\xi}}\mathcal{GI}(\xi),\mathcal{GI}(\xi))$ and $({^{\bot_\xi}}\mathcal{I}(\xi)_n\cap{^{\bot_\xi}}\mathcal{GI}(\xi),\mathcal{GI}(\xi)_n)$ are complete hereditary cotorsion pairs in $\mathrm{D}(R)$ with respect to $\xi$.

$(5)$ If $R$ is a virtually Gorenstein ring, then  $(\mathcal{GP}(\xi),\mathcal{GP}(\xi)^{\bot_\xi})$ and $(\mathcal{GP}(\xi)_n,\mathcal{P}(\xi)^{\bot_\xi}_n\cap\mathcal{GP}(\xi)^{\bot_\xi})$ are complete hereditary cotorsion pairs in $\mathrm{D}(R)$ with respect to $\xi$.}}

\bigskip
\section{\bf Preliminaries}
Throughout $R$ is an associated ring with identity. This section collects some notions and facts about complexes, triangulated categories and cotorsion pairs.

\vspace{2mm}
{\bf Complexes.} We denote by $\mathrm{C}(R)$ the category whose class of objects consists of all
cochain complexes $X$ of $R$-modules,
$$\cdots\longrightarrow X^{s-1}\stackrel{d^{s-1}_X}\longrightarrow X^s\stackrel{d^{s}_X}\longrightarrow X^{s+1}\longrightarrow\cdots,$$such that $d^{s}_Xd^{s-1}_X=0$ for all $s\in\mathbb{Z}$.
Denotes by $\mathrm{K}(R)$
the homotopy category of complexes of $R$-modules and by $\mathrm{D}(R)$ the derived category of complexes of $R$-modules. An isomorphism in $\mathrm{D}(R)$ is marked by $\simeq$. A quasi-isomorphism is marked by a $\simeq$ next to the arrow.
  Let $X$ be a complex and $s\in\mathbb{Z}$. The \emph{$s$-th shift} of $X$ is
denoted  by $X[s]$. Also denote $$\mathrm{Z}^s(X)=\mathrm{ker}d_X^{s},\ \mathrm{B}^{s}(X)=\mathrm{im}d_X^{s-1},\ \mathrm{C}^{s}(X)=\mathrm{coker}d_X^{s-1},\ \mathrm{H}^s(X)=\mathrm{Z}^s(X)/\mathrm{B}^s(X).$$
The complex $X$ is called \emph{contractible} if the identity morphism $\mathrm{id}_X$ is null-homotopic. The \emph{cone} of a map $f:X\rightarrow Y$ in $\mathrm{C}(R)$ is the complex such that $\mathrm{cn}(f)^n=X^{n+1}\oplus Y^n$ and
$d^n_{\mathrm{cn}(f)}(x,y)=(-d^{n+1}_X(x),f^{n+1}(x)+d^{n}_Y(x))$ for each $n\in\mathbb{Z}$.

An $R$-complex $P$ is called \emph{projective} if the functor
$\mathrm{Hom}_{\mathrm{C}(A)}(P,-)$ is exact. An $R$-complex $P$ is called \emph{semi-projective} if $\mathrm{Hom}_R(P,\beta)$ is a
surjective quasi-isomorphism for every surjective quasi-isomorphism $\beta$ in $\mathrm{C}(R)$. A \emph{semi-projective resolution} of an $R$-complex $X$ is a
quasi-isomorphism $P\stackrel{\simeq}\rightarrow X$ of $R$-complexes with $P$ semi-projective. Dually, one can define \emph{injective complex}, \emph{semi-injective} complexes and \emph{semi-injective resolution}. It is well known that every $R$-complex admits both a semi-projective and a semi-injective resolution (see \cite{CFH}).

Let $\textsf{X}$ be a class of $R$-modules. Denote $\mathrm{C}(\textsf{X})$ (resp. $\mathrm{K}(\textsf{X})$) the
class (resp. the homotopy category) of complexes with components in $\textsf{X}$.

\begin{lem}\label{lem4.1} {\rm (\cite[Theorem 3.1]{O}).} Let $\textsf{X}$ be a class of $R$-modules and $X\in\mathrm{K}(\textsf{X})$, and let $L,M\in\mathrm{Mod}R$ be such that $\mathrm{Ext}^{i\geq1}_{R}(L,\textsf{X})=0=\mathrm{Ext}^{i\geq1}_{R}(\textsf{X},M)$. Then for $s\in\mathbb{Z}$, there are long exact sequences
\begin{equation}
\begin{split}
0&\rightarrow\mathrm{Ext}^1_{R}(\mathrm{C}^{s+1}(X),M)\rightarrow\mathrm{H}^{-s}(\mathrm{Hom}_{R}(X,M))
 \rightarrow\mathrm{Hom}_{R}(\mathrm{H}^{s}(X),M)\\
&\rightarrow\mathrm{Ext}^2_{R}(\mathrm{C}^{s+1}(X),M)
 \rightarrow\mathrm{Ext}^1_{R}(\mathrm{C}^{s}(X),M)\rightarrow\mathrm{Ext}^1_{R}(\mathrm{H}^{s}(X),M)\\
&\rightarrow\mathrm{Ext}^3_{R}(\mathrm{C}^{s+1}(X),M)\rightarrow
 \mathrm{Ext}^2_{R}(\mathrm{C}^{s}(X),M)\rightarrow\mathrm{Ext}^2_{R}(\mathrm{H}^{s}(X),M)\rightarrow\cdots,
\end{split}
\end{equation}

\vspace{0.01mm}
\begin{equation}
\begin{split}
0&\rightarrow\mathrm{Ext}^1_{R}(L,\mathrm{Z}^{s-1}(X))\rightarrow\mathrm{H}^{s}(\mathrm{Hom}_{R}(L,X))
 \rightarrow\mathrm{Hom}_{R}(L,\mathrm{H}^{s}(X))\\
&\rightarrow\mathrm{Ext}^2_{R}(L,\mathrm{Z}^{s-1}(X))
 \rightarrow\mathrm{Ext}^1_{R}(L,\mathrm{Z}^{s}(X))\rightarrow\mathrm{Ext}^1_{R}(L,\mathrm{H}^{s}(X))\\
&\rightarrow\mathrm{Ext}^3_{R}(L,\mathrm{Z}^{s-1}(X))\rightarrow
 \mathrm{Ext}^2_{R}(L,\mathrm{Z}^{s}(X))\rightarrow\mathrm{Ext}^2_{R}(L,\mathrm{H}^{s}(X))\rightarrow\cdots.
\end{split}
\end{equation}
 \end{lem}

 \begin{lem}\label{lem0.0} {\rm (\cite[Lemma 5.2]{E}).} Given a short exact sequence $0\rightarrow X\rightarrow Y\rightarrow Z\rightarrow 0$ in $\mathrm{C}(R)$, the following statements are equivalent:

$(1)$ $0\rightarrow \mathrm{Z}^s(X)\rightarrow \mathrm{Z}^s(Y)\rightarrow \mathrm{Z}^s(Z)\rightarrow 0$ is exact for all $s\in\mathbb{Z}$;

$(2)$ $0\rightarrow \mathrm{C}^s(X)\rightarrow \mathrm{C}^s(Y)\rightarrow \mathrm{C}^s(Z)\rightarrow 0$ is exact for all $s\in\mathbb{Z}$;

$(3)$ $0\rightarrow \mathrm{H}^s(X)\rightarrow \mathrm{H}^s(Y)\rightarrow \mathrm{H}^s(Z)\rightarrow 0$ is exact for all $s\in\mathbb{Z}$.
\end{lem}

An exact sequence $0\rightarrow X\rightarrow Y\rightarrow Z\rightarrow 0$ in $\mathrm{C}(R)$ is called \emph{CE exact} if the above equivalent conditions hold.  Given a class $\textsf{X}$ of $R$-modules, a complex $X$ is called a \emph{CE $\textsf{X}$
complex} if $X$, $\mathrm{Z}(X)$, $\mathrm{B}(X)$ and $\mathrm{H}(X)$ are all in $\mathrm{C}(\textsf{X})$.
In particular, if $\textsf{X}$ is the class $\textsf{P}$ (resp. $\textsf{I}$, $\textsf{F}$) of
projective (resp. injective, flat) $R$-modules, then a CE $\textsf{X}$ complex is just a CE projective (resp. injective, flat) complex.
By \cite[Proposition 6.3]{E}, the functor $\mathrm{Hom}(-,-)$ on
$\mathrm{C}(R)\times\mathrm{C}(R)$ is right balanced by $\textrm{CE-\textsf{P}}\times\textrm{CE-\textsf{I}}$. So one can compute derived functors of $\mathrm{Hom}(-,-)$ using either of the two resolutions, and denote these functors by $\overline{\mathrm{Ext}}^n(-,-)$.

 \begin{lem}\label{lem0.0'} {\rm (\cite{E,EFHO}).}
$(1)$ An $R$-complex $P$ is CE projective if and only if $P=P'\oplus P''$ where $P'$ is a projective
complex and $P''=\bigoplus_{n\in\mathbb{Z}}P^n[-n]$ for $P^n\in\textsf{P}$.

$(2)$ An $R$-complex $I$ is CE injective if and only if $I=I'\oplus I''$ where $I'$ is an injective
complex and $I''=\bigoplus_{n\in\mathbb{Z}}I^n[-n]$ for $I^n\in\textsf{I}$.
\end{lem}

\vspace{2mm}
{\bf Triangulated categories.} Let $\mathcal{T}=(\mathcal{T},[1],\triangle)$ be a triangulated category, where $\mathcal{T}$ is an additive category, $[1]$
is an autoequivalence of $\mathcal{T}$. The elements of $\triangle$ are diagrams in $\mathcal{T}$ of the form $X\rightarrow Y\rightarrow Z\rightarrow X[1]$, called the \emph{exact triangles}, satisfying axioms (Tr1)-(Tr4) of \cite[Definitions 1.1.1 and 1.3.13]{N}.
Let $\xi\subseteq\triangle$ be a proper class of triangles (detail see \cite{B}). Following \cite{B},
an object $P$ in $\mathcal{T}$ is called \emph{$\xi$-projective} (resp. \emph{$\xi$-injective})
if for any triangle $X\rightarrow Y\rightarrow Z\rightarrow X[1]$ in $\xi$, the induced sequence$$0\rightarrow\mathrm{Hom}_{\mathcal{T}}(P,X)\rightarrow\mathrm{Hom}_{\mathcal{T}}(P,Y)
\rightarrow\mathrm{Hom}_{\mathcal{T}}(P,Z)\rightarrow0$$$$(\textrm{resp.}\ 0\rightarrow\mathrm{Hom}_{\mathcal{T}}(Z,P)\rightarrow\mathrm{Hom}_{\mathcal{T}}(Y,P)
\rightarrow\mathrm{Hom}_{\mathcal{T}}(X,P)\rightarrow0)$$ is exact in the category Ab of abelian groups. The symbol $\mathcal{P}(\xi)$ (resp. $\mathcal{I}(\xi)$) denote the full subcategory of $\xi$-projective
(resp. $\xi$-injective) objects in $\mathcal{T}$. The category $\mathcal{T}$ is said to \emph{have
enough $\xi$-projectives} (resp. \emph{$\xi$-injectives}) if for any $X\in\mathcal{T}$, there is a triangle
$$K\rightarrow P\rightarrow X\rightarrow K[1]\ (\textrm{resp.}\ X\rightarrow I\rightarrow C\rightarrow X[1])$$ in $\xi$ with $P\in\mathcal{P}(\xi)$ (resp. $I\in\mathcal{I}(\xi)$).

Let $\mathcal{X}$ be a full additive subcategory of $\mathcal{T}$, a morphism $f$ in $\mathcal{T}$ is called an \emph{$\mathcal{X}$-ghost} (resp. \emph{$\mathcal{X}$-coghost}) if $\mathrm{Hom}_{\mathcal{T}}(X,f)=0$ (resp. $\mathrm{Hom}_{\mathcal{T}}(f,X)=0$) for any $X\in\mathcal{X}$ (see \cite{CW}). Let $C$ be an object of $\mathcal{T}$. For any $n\neq0$, the \emph{$\xi$-extension functor} $\xi xt^{n}_\xi(-,C)$ is
defined to be the $n$th right $\xi$-derived functor of the functor $\mathrm{Hom}_\mathcal{T}(-,C)$, that is
$$\xi xt^{n}_\xi(-,C):=\mathcal{R}^n_\xi\mathrm{Hom}_\mathcal{T}(-,C).$$

\vspace{2mm}
{\bf Cotorsion pairs.}  A pair of classes $(\mathcal{A},\mathcal{B})$ in a triangulated category $\mathcal{T}$ is called
a \emph{cotorsion pair with respect to $\xi$} if
$$\mathcal{A}^{\perp_\xi}:=\{U\in\mathcal{T}\hspace{0.03cm}|\hspace{0.03cm}\xi xt^{1}_\xi(A,U)=0\ \textrm{for\ all}\ A\in\mathcal{A}\}=\mathcal{B},$$ $${^{\perp_\xi}}\mathcal{B}:=\{V\in\mathcal{T}\hspace{0.03cm}|\hspace{0.03cm}\xi xt^{1}_\xi(V,B))=0\ \textrm{for\ all}\ B\in\mathcal{B}\}=\mathcal{A}.$$
A cotorsion pair $(\mathcal{A},\mathcal{B})$ with respect to $\xi$ is called \emph{complete} if for $X\in\mathcal{T}$ there are triangles $B\rightarrow A\rightarrow X\rightarrow B[1]$ and $X\rightarrow B'\rightarrow A'\rightarrow X[1]$
in $\xi$ with $A,A'\in\mathcal{A}$ and $B,B'\in\mathcal{B}$. The cotorsion pair $(\mathcal{A},\mathcal{B})$ is called \emph{hereditary} if for any triangles $A_1\rightarrow A_2\rightarrow A_3\rightarrow A_1[1]$ and  $B_1\rightarrow B_2\rightarrow B_3\rightarrow B_1[1]$ in $\xi$ with $A_2,A_3\in\mathcal{A}$ and $B_1,B_2\in\mathcal{B}$, one has $A_1\in\mathcal{A}$ and $B_3\in\mathcal{B}$.

 The category of (left) $R$-modules is denoted by $\textrm{Mod}R$. A \emph{cotorsion pair} in $\textrm{Mod}R$ is a pair $(\textsf{A},\textsf{B})$ of $R$-modules
 such that
$$\textsf{A}^\perp:=\{M\in\textrm{Mod}R\hspace{0.03cm}|\hspace{0.03cm}\textrm{Ext}^1_R(A,M)=0\ \textrm{for\ all}\ A\in\textsf{A}\}=\textsf{B},$$ $${^\perp}\textsf{B}:=\{N\in\textrm{Mod}R\hspace{0.03cm}|\hspace{0.03cm}\textrm{Ext}^1_R(N,B)=0\ \textrm{for\ all}\ B\in\textsf{B}\}=\textsf{A}.$$
A cotorsion pair $(\textsf{A},\textsf{B})$ is called \emph{complete} if
for $M\in\textrm{Mod}R$, there are exact
sequences $0\rightarrow B\rightarrow A\rightarrow M\rightarrow 0$ and $0\rightarrow M\rightarrow B'\rightarrow A'\rightarrow 0$ with $A,A'\in\textsf{A}$ and $B,B'\in\textsf{B}$.
 A cotorsion pair $(\textsf{A},\textsf{B})$ is called \emph{hereditary} if $\textsf{A}$
is closed under taking kernels of
epimorphisms between objects in $\textsf{A}$ and $\textsf{B}$ is closed under taking cokernels of monomorphisms
between objects in $\textsf{B}$.

Let
$(\textsf{A},\textsf{B})$ be a cotorsion pair in $\textrm{Mod}R$. Following \cite{G}, a complex
 $X$ is called an \emph{$\textsf{A}$ complex} if it is exact and each
$\mathrm{Z}^n(X)\in\textsf{A}$. The definition of \emph{$\textsf{B}$ complex} is similar. These classes of complexes are denoted by $\widetilde{\textsf{A}}$ and $\widetilde{\textsf{B}}$, respectively. A complex $X$ is called a \emph{$dg\textsf{A}$ complex} if each $X^n\in\textsf{A}$, and $\mathrm{Hom}_R(X,B)$ is exact whenever $B\in\widetilde{\textsf{B}}$. The definition of \emph{$dg\textsf{B}$ complex} is dual. These classes of complexes are denoted by $dg\widetilde{\textsf{A}}$ and $dg\widetilde{\textsf{B}}$, respectively.
 By \cite[Theorem 3.5]{YL11} or \cite[Theorem 2.5]{YD}, if the cotorsion pair $(\textsf{A},\textsf{B})$ in $\textrm{Mod}R$ is complete and hereditary then the induced cotorsion pairs $(dg\widetilde{\textsf{A}},\widetilde{\textsf{B}})$ and $(\widetilde{\textsf{A}},dg\widetilde{\textsf{B}})$ in $\mathrm{C}(R)$ are complete and hereditary.

\begin{rem}\label{lem:2.2}{\rm (1) For a class $\textsf{X}$ of $R$-modules, let $\langle\textsf{X}\rangle_1$ be the collection of
all retracts of coproducts of shifts of elements in $\textsf{X}$. There is a proper class $\xi(\langle\textsf{P}\rangle_1)$ of triangles in $\mathrm{D}(R)$ as follows:
A triangle $X\rightarrow Y\rightarrow Z\rightarrow X[1]$ is in $\xi(\langle\textsf{P}\rangle_1)$ if and only if
 the sequence $0\rightarrow\mathrm{Hom}_{\mathrm{D}(R)}(P,X)\rightarrow\mathrm{Hom}_{\mathrm{D}(R)}(P,Y)\rightarrow\mathrm{Hom}_{\mathrm{D}(R)}(P,Z)\rightarrow0$ is exact for all $P\in\langle\textsf{P}\rangle_1$. \textbf{In what follows, $\xi$ always denote the proper class $\xi(\langle\textsf{P}\rangle_1)$}.
 The class $\mathcal{P}(\xi)$ based on $\xi$ is exactly $\langle\textsf{P}\rangle_1$ and  the class $\mathcal{I}(\xi)$ based on $\xi$ is exactly $\langle\textsf{I}\rangle_1$.

(2) A $\langle\textsf{P}\rangle_1$-ghost or $\langle\textsf{I}\rangle_1$-coghost morphism $f$ is called a \emph{ghost} in \cite{HL} or \cite{EFHO}. Equivalently, $\mathrm{H}^s(f)=0$ for all $s\in\mathbb{Z}$. A triangle $X\rightarrow Y\rightarrow Z\rightarrow X[1]$ in $\xi$ is called a \emph{cohomologically ghost triangle} in \cite{O}. Equivalently, $0\rightarrow\mathrm{H}^s(X)\rightarrow\mathrm{H}^s(Y)\rightarrow\mathrm{H}^s(Z)\rightarrow0$ is exact for all $s\in\mathbb{Z}$.

(3) Following Hovey and Lockridge \cite{HL11}, a complex $X$ is called \emph{$\xi$-flat} if every ghost $f$ with domain $X$ is \emph{phantom}, in the sense that
$\mathrm{Hom}_{\mathrm{D}(R)}(T,f)=0$ for all compact objects $T$ in $\mathrm{D}(R)$. Equivalently, $X\in\langle\textsf{F}\rangle_1$ by \cite[Proposition 2.3]{HL11}.

(4) For any complex $X$,
there exists two CE exact sequences $0\rightarrow K\rightarrow P\rightarrow X\rightarrow 0$ and $0\rightarrow X\rightarrow I\rightarrow C\rightarrow 0$ in $\mathrm{C}(R)$ with $P\in\textrm{CE-\textsf{P}}$ and $I\in\textrm{CE-\textsf{I}}$, the induced triangles $$K\rightarrow P\rightarrow X\rightarrow K[1],\ X\rightarrow I\rightarrow C\rightarrow X[1]$$ in $\mathrm{D}(R)$ are cohomologically ghost triangles with $P\in\langle\textsf{P}\rangle_1$ and $I\in\langle\textsf{I}\rangle_1$.}
\end{rem}

\section{\bf Induced complete hereditary cotorsion pairs in $\mathrm{D}(R)$}

This section is devoted to proving a proof of Theorem A stated in introduction. We begin with the following several useful lemmas.

\begin{lem}\label{lem2.3} Let $X$ and $Y$ be two complexes. The following are equivalent:

$(1)$ $\overline{\mathrm{Ext}}^1(P,Y)=0$, where $P\stackrel{\simeq}\rightarrow X$ is a semi-projective resolution;

$(2)$ $\xi xt^{1}_\xi(X,Y)=0$;

$(3)$ $\overline{\mathrm{Ext}}^1(X,I)=0$, where $Y\stackrel{\simeq}\rightarrow I$ is a semi-injective resolution;
\end{lem}
\begin{proof} Consider the CE exact sequence $0\rightarrow Y\rightarrow E\rightarrow L\rightarrow0$ in $\mathrm{C}(R)$ with $E\in\textrm{CE-\textsf{I}}$. Then $\overline{\mathrm{Ext}}^1(P,\mathrm{cn}(\mathrm{id}_{E[-1]}))=0=\overline{\mathrm{Ext}}^1(P,E)$ and $\xi xt^{1}_\xi(X,E)=0$ by \cite[Proposition 4.17]{B}. As $\mathrm{Hom}_{\mathrm{D}(R)}(P,-)\cong\mathrm{Hom}_{\mathrm{K}(R)}(P,-)$,  one has the following commutative diagram\begin{center} $\xymatrix@C=16pt@R=16pt{
\mathrm{Hom}_{\mathrm{C}(R)}(P,\mathrm{cn}(\mathrm{id}_{E[-1]}))\ar[d]\ar[r] &\mathrm{Hom}_{\mathrm{C}(R)}(P,\mathrm{cn}(\mathrm{id}_{L[-1]}))\ar[d]\ar[r]&\overline{\mathrm{Ext}}^1(P,\mathrm{cn}(\mathrm{id}_{Y[-1]})) \ar[d]\ar[r]&0\\
\mathrm{Hom}_{\mathrm{C}(R)}(P,E)\ar[d]^\alpha\ar[r] &\mathrm{Hom}_{\mathrm{C}(R)}(P,L)\ar[d]^\beta\ar[r]&\overline{\mathrm{Ext}}^1(P,Y) \ar[d]^\gamma\ar[r]&0\\
\mathrm{Hom}_{\mathrm{D}(R)}(P,E)\ar[d]\ar[r] &\mathrm{Hom}_{\mathrm{D}(R)}(P,L)\ar[d]\ar[r]&\xi xt^{1}_\xi(P,Y)\ar[r]&0\\ 0 &0 }$
\end{center}by \cite[Theorem 2.3.10]{CFH} and \cite[Corollary 4.12]{B}, it implies that $\gamma$ is surjective.

(1) $\Leftrightarrow$ (2) If $\overline{\mathrm{Ext}}^1(P,Y)=0$, then $\xi xt^{1}_\xi(X,Y)\cong\xi xt^{1}_\xi(P,Y)=0$. Conversely, it follows by \cite[Lemma 2.1]{G} that
$\overline{\mathrm{Ext}}^1(P,\mathrm{cn}(\mathrm{id}_{Y[-1]}))\subseteq\mathrm{Ext}^1_R(P,\mathrm{cn}(\mathrm{id}_{Y[-1]}))=0$, so $\overline{\mathrm{Ext}}^1(P,Y)=0$.

 (2) $\Leftrightarrow$ (3) The proof is similar to  (1) $\Leftrightarrow$ (2).
\end{proof}

\begin{lem}\label{lem2.1} Let $(\textsf{A},\textsf{B})$ be a hereditary cotorsion pair in $\textrm{Mod}R$ and $X\in\mathrm{D}(R)$. Then the following are equivalent:

$(1)$ $X\in\mathcal{A}_\textsf{P}$;

$(2)$ $\xi xt^{1}_\xi(X,Y)=0$ for  all $Y\in\mathrm{CE}\textrm{-}\textsf{B}$;

$(3)$ $\mathrm{H}^s(X)\in\textsf{A}$ and $\mathrm{H}^{-s}(\mathrm{RHom}_R(X,B))\cong\mathrm{Hom}_R(\mathrm{H}^s(X),B)$ for all $B\in\textsf{B}$ and $s\in\mathbb{Z}$;

$(4)$ $\mathrm{C}^s(P)\in\textsf{A}$ for any semi-projective resolution $P\stackrel{\simeq}\rightarrow X$  and all $s\in\mathbb{Z}$;

$(5)$ $\overline{\mathrm{Ext}}^1(P,Y)=0$ for all $Y\in\mathrm{CE}\textrm{-}\textsf{B}$, where $P\stackrel{\simeq}\rightarrow X$ is a semi-projective resolution.
\end{lem}
\begin{proof} (1) $\Rightarrow$ (2) Let $P\stackrel{\simeq}\rightarrow X$ be a semi-projective resolution with $\mathrm{C}^s(P)\in\textsf{A}$ for all $s\in\mathbb{Z}$ and $Y\in\mathrm{CE}\textrm{-}\textsf{B}$. As $P\in\mathrm{CE}\textrm{-}\textsf{A}$,  it follows by  \cite[Theorem 9.4]{E} that $\overline{\mathrm{Ext}}^1(P,Y)=0$,  so $\xi xt^{1}_\xi(X,Y)=0$ by Lemma \ref{lem2.3}.

(2) $\Rightarrow$ (3) Let $X\in\mathrm{D}(R)$ be such that $\xi xt^{1}_\xi(X,Y)=0$ for all $Y\in\mathrm{CE}\textrm{-}\textsf{B}$ and $P\stackrel{\simeq}\rightarrow X$ a semi-projective resolution. Then $\mathrm{Ext}^1_R(\mathrm{C}^s(P),B)\cong\overline{\mathrm{Ext}}^1(P,B[-s])=0$ for all $B\in\textsf{B}$ by \cite[Lemma 9.3]{E} and Lemma \ref{lem2.3} as $B[-s]\in\mathrm{CE}\textrm{-}\textsf{B}$, and so $\mathrm{C}^s(P)\in\textsf{A}$ for all $s\in\mathbb{Z}$. Hence Lemma \ref{lem4.1}(1.1) implies that $\mathrm{H}^s(X)\in\textsf{A}$ and $\mathrm{H}^{-s}(\mathrm{RHom}_R(X,B))\cong\mathrm{H}^{-s}(\mathrm{Hom}_R(P,B))\cong\mathrm{Hom}_R(\mathrm{H}^s(X),B)$ for all $B\in\textsf{B}$ and $s\in\mathbb{Z}$.

 (3) $\Rightarrow$ (4) Let $P\stackrel{\simeq}\rightarrow X$ be a semi-projective resolution. Then $\mathrm{C}^s(P)\in\textsf{A}$ by Lemma \ref{lem4.1}(1.1) for all $s\in\mathbb{Z}$, so $X\in\mathcal{A}_\textsf{P}$.

  (4) $\Rightarrow$ (1) is trivial and (2) $\Leftrightarrow$ (5) follows by Lemma \ref{lem2.3}.
\end{proof}

Using similar arguments and Lemma \ref{lem4.1}(1.2), one can prove the next dual lemma.

\begin{lem}\label{lem2.2} Let $(\textsf{A},\textsf{B})$ be a hereditary cotorsion pair in $\textrm{Mod}R$ and $Y\in\mathrm{D}(R)$. Then the following are equivalent:

$(1)$ $Y\in\mathcal{B}_\textsf{I}$;

$(2)$ $\xi xt^{1}_\xi(X,Y)=0$ for  all $X\in\mathrm{CE}\textrm{-}\textsf{A}$;

$(3)$ $\mathrm{H}^s(Y)\in\textsf{A}$ and $\mathrm{H}^{s}(\mathrm{RHom}_R(A,Y))\cong\mathrm{Hom}_R(A,\mathrm{H}^s(X))$ for all $A\in\textsf{A}$ and $s\in\mathbb{Z}$;

$(4)$ $\mathrm{Z}^s(I)\in\textsf{A}$ for any semi-injective resolution $Y\stackrel{\simeq}\rightarrow I$ and all $s\in\mathbb{Z}$;

$(5)$ $\overline{\mathrm{Ext}}^1(X,I)=0$ for all $X\in\mathrm{CE}\textrm{-}\textsf{A}$, where $Y\stackrel{\simeq}\rightarrow I$ is a semi-injective resolution.
\end{lem}

\begin{lem}\label{lem2.23} Let $(\textsf{A},\textsf{B})$ be a hereditary cotorsion pair in $\textrm{Mod}R$. Then $\mathrm{CE}\textrm{-}\textsf{A}\subseteq\mathcal{A}_\textsf{P}$ and $\mathrm{CE}\textrm{-}\textsf{B}\subseteq\mathcal{B}_\textsf{I}$ in $\mathrm{D}(R)$.
\end{lem}
\begin{proof} Let $Y\in\mathrm{CE}\textrm{-}\textsf{B}$, and let $A\in\textsf{A}$ and $P_A\rightarrow A$ be a projective resolution. Set $U:=\mathrm{Ker}(P_A\rightarrow A)$ and $Y_{\supset_n}:0\rightarrow\mathrm{C}^{n}(Y)\rightarrow Y^{n+1}\rightarrow Y^{n+2}\rightarrow\cdots$. Then $\{Y_{\supset_n}\twoheadrightarrow Y_{\supset_{n+1}}\}_{n\geq0}$ is an inverse system of epimorphisms in $\mathrm{CE}\textrm{-}\textsf{B}$
and $Y\cong\underleftarrow{\textrm{lim}}Y_{\supset_n}$. Set $K_{n,n+1}:=\mathrm{Ker}(Y_{\supset_n}\rightarrow Y_{\supset_{n+1}})$. Then $K_{n,n+1}\in\mathrm{CE}\textrm{-}\textsf{B}$. For $n\geq0$, the induced sequence $$0\rightarrow\mathrm{Hom}_{R}(U,K_{n,n+1})\rightarrow\mathrm{Hom}_{R}(U,Y_{\supset_n})\rightarrow
\mathrm{Hom}_{R}(U,Y_{\supset_{n+1}})\rightarrow0$$ is exact. Thus $\{\mathrm{Hom}_{R}(U,Y_{\supset_{n}})\twoheadrightarrow\mathrm{Hom}_{R}(U,Y_{\supset_{n+1}})\}$ is an inverse system of epimorphisms. As $U\in\mathrm{CE}\textrm{-}\textsf{A}$ is exact, $\mathrm{Hom}_{R}(U,Y_{\supset_{n}})$ is exact by \cite[Lemma 2.5]{CFHo}, it implies that
 $\mathrm{Hom}_R(U,Y)\cong\underleftarrow{\textrm{lim}}\mathrm{Hom}_R(U,Y_{\supset_n})$ is exact. Thus $$\mathrm{RHom}_R(A,Y)\simeq\mathrm{Hom}_R(P_A,Y)\simeq\mathrm{Hom}_R(A,Y),$$ so $\mathrm{H}^s(\mathrm{RHom}_R(A,Y))\cong\mathrm{Hom}_R(A,\mathrm{H}^s(Y))$ for all $s\in\mathbb{Z}$ by Lemma \ref{lem4.1}(1.1), and hence $Y\in\mathcal{B}_\textsf{I}$ by Lemma \ref{lem2.2}. Similarly, one can show that $\mathrm{CE}\textrm{-}\textsf{A}\subseteq\mathcal{A}_\textsf{P}$.
\end{proof}

With the above preparations, now we prove Theorem A.

\begin{thm}\label{lem2.5} The following are equivalent:

$(1)$ $(\textsf{A},\textsf{B})$ is a complete hereditary cotorsion pair in $\mathrm{Mod}R$;

 $(2)$ $(\mathcal{A}_\textsf{P},\mathcal{B}_\textsf{I})$ is a complete hereditary cotorsion pair in $\mathrm{D}(R)$ with respect to $\xi$.
\end{thm}
\begin{proof}  (1) $\Rightarrow$ (2) Let $X\in{^{\perp_\xi}}\mathcal{B}_\textsf{I}$ and $Q\stackrel{\simeq}\rightarrow X$ be a semi-projective resolution, and let  $Y\in\mathrm{CE}\textrm{-}\textsf{B}$. Then $Y\in\mathcal{B}_\textsf{I}$ by Lemma \ref{lem2.23}, so
 $\overline{\mathrm{Ext}}^1(Q,Y)=0$  by Lemma \ref{lem2.3} and hence $Q\in\mathrm{CE}\textrm{-}\textsf{A}$ by \cite[Theorem 9.4]{E}. Hence Lemma \ref{lem2.1} implies that $X\in\mathcal{A}_\textsf{P}$. Let $X\in\mathcal{A}_\textsf{P}$ and $P\stackrel{\simeq}\rightarrow X$ be a semi-projective resolution with $P\in\mathrm{CE}\textrm{-}\textsf{A}$. Then $\xi xt^{1}_\xi(P,Y)=0$ for all $Y\in\mathcal{B}_\textsf{I}$ by Lemma \ref{lem2.2}, so $X\in{^{\perp_\xi}}\mathcal{B}_\textsf{I}$. Similarly, one has $Y\in\mathcal{B}_\textsf{I}$ if and only if $Y\in\mathcal{A}^{\perp_\xi}_\textsf{P}$.
 Thus $(\mathcal{A}_\textsf{P},\mathcal{B}_\textsf{I})$ is a cotorsion pair in $\mathrm{D}(R)$ with respect to $\xi$. Let $C\in\mathrm{D}(R)$. Consider the exact sequences $0\rightarrow \bar{Y}\rightarrow \bar{X}\rightarrow C\rightarrow0$ and $0\rightarrow \bar{X}\rightarrow \tilde{Y}\rightarrow \tilde{X}\rightarrow0$ in $\mathrm{C}(R)$ with $\bar{X}\in\mathrm{dg}\widetilde{\textsf{A}}$, $\tilde{Y}\in\mathrm{dg}\widetilde{\textsf{B}}$ and $\bar{Y}\in\widetilde{\textsf{B}}$, $\tilde{X}\in\widetilde{\textsf{A}}$ by \cite[Theorem 2.4]{YD}. Then each $\tilde{Y}^s\in\textsf{A}\cap\textsf{B}$.
For $s\in\mathbb{Z}$, there exists an exact sequence $0\rightarrow K^s\rightarrow X^s\stackrel{f^s}\rightarrow \mathrm{Z}^s(\tilde{Y})\rightarrow0$ with $X^s\in\textsf{A}$ and $K^s\in\textsf{B}$. As $\mathrm{Hom}_R(X^s,\mathrm{Z}^s(\tilde{Y}))\cong\mathrm{Hom}_{\mathrm{C}(R)}(X^s[-s],\tilde{Y})$, it induces a cochain map $f:\bigoplus_{s\in\mathbb{Z}}X^s[-s]\rightarrow \tilde{Y}$. Set $X=\bigoplus_{s\in\mathbb{Z}}X^s[-s]\oplus\mathrm{cn}(\mathrm{id}_{\tilde{Y}[-1]})$ and $p=(0,\mathrm{id}_{\tilde{Y}^s})_{s\in\mathbb{Z}}:\mathrm{cn}(\mathrm{id}_{\tilde{Y}[-1]})\rightarrow \tilde{Y}$. We have the following CE exact sequence $$0\longrightarrow Y\longrightarrow X\stackrel{[f,p]}\longrightarrow \tilde{Y}\longrightarrow0,$$such that $X\in\mathrm{CE}\textrm{-}\textsf{A}$ and $Y\in\mathrm{CE}\textrm{-}\textsf{B}$.
Hence
 $$Y\rightarrow X\rightarrow C\rightarrow Y[1]$$ is a cohomologically ghost triangle in $\mathrm{D}(R)$ with $X\in\mathcal{A}_\textsf{P}$ and $Y\in\mathcal{B}_\textsf{I}$ by Lemma \ref{lem2.23}. On the other hand, there is a cohomologically ghost triangle $C\rightarrow I\rightarrow D\rightarrow C[1]$ in $\mathrm{D}(R)$ with $I\in\langle\textsf{I}\rangle_1$, and there is a cohomologically ghost triangle $Y_D\rightarrow X'\rightarrow D\rightarrow Y_D[1]$ in $\mathrm{D}(R)$ with $X'\in\mathcal{A}_\textsf{P}$ and $Y_D\in\mathcal{B}_\textsf{I}$.  We have the next commutative diagram of exact triangles:\begin{center} $\xymatrix@C=18pt@R=16pt{
    &Y_D\ar[d] \ar@{=}[r] & Y_D \ar[d] \\
  C \ar@{=}[d]\ar[r]& Y'\ar[d]\ar[r] &X'\ar[d]\ar[r]&C[1]\ar@{=}[d] \\
 C\ar[r]& I\ar[d]\ar[r] &D\ar[d]\ar[r]&C[1] \\
    & Y_D[1]\ar@{=}[r]& Y_D[1] &}$
\end{center}As $Y_D,I\in\mathcal{B}_\textsf{I}$, $Y'\in\mathcal{B}_\textsf{I}$, it implies that $$C\rightarrow Y'\rightarrow X'\rightarrow C[1]$$ is the desired cohomologically ghost triangle in $\mathrm{D}(R)$ with $X'\in\mathcal{A}_\textsf{P}$ and $Y'\in\mathcal{B}_\textsf{I}$.
Thus $(\mathcal{A}_\textsf{P},\mathcal{B}_\textsf{I})$ is a complete cotorsion pair in $\mathrm{D}(R)$ with respect to $\xi$. Let $X_1\rightarrow X_2\rightarrow X_3\rightarrow X_1[1]$ be a cohomologically ghost triangle in $\mathrm{D}(R)$ with $X_2,X_3\in\mathcal{A}_\textsf{P}$. Then $0\rightarrow\mathrm{H}^s(X_1)\rightarrow \mathrm{H}^s(X_2)\rightarrow \mathrm{H}^s(X_3)\rightarrow0$ is exact and $\mathrm{H}^s(X_2),\mathrm{H}^s(X_3)\in\textsf{A}$, so $\mathrm{H}^s(X_1)\in\textsf{A}$ for all $s\in\mathbb{Z}$. Also for $s\in\mathbb{Z}$ and $B\in\textsf{B}$, there is a commutative diagram:
\begin{center} $\xymatrix@C=18pt@R=16pt{
& \mathrm{H}^{-s}(\mathrm{RHom}_R(X_3,B)) \ar[d]^\cong\ar[r]& \mathrm{H}^{-s}(\mathrm{RHom}_R(X_2,B))\ar[d]^\cong\ar[r] &\mathrm{H}^{-s}(\mathrm{RHom}_R(X_1,B))\ar[d]& \\
0\ar[r]&\mathrm{Hom}_R(\mathrm{H}^s(X_3),B) \ar[r]& \mathrm{Hom}_R(\mathrm{H}^s(X_2),B)\ar[r] &\mathrm{Hom}_R(\mathrm{H}^s(X_1),B)\ar[r]&0 }$
\end{center}it follows that the above row is exact and $X_1\in\mathcal{A}_\textsf{P}$ by Lemma \ref{lem2.1}. By the dual argument, it implies that the cotorsion pair $(\mathcal{A}_\textsf{P},\mathcal{B}_\textsf{I})$ in $\mathrm{D}(R)$ is hereditary with respect to $\xi$.

 (2) $\Rightarrow$ (1) Let $A\in\textsf{A}$ and $B\in\textsf{B}$, and let $B\rightarrow J_B$ be an injective resolution. Then $J_B\in\mathrm{CE}\textrm{-}\textsf{B}$ and $\xi xt^{1}_\xi(A,B)=0$ by Lemma \ref{lem2.23}. Thus $\mathrm{Ext}^1_R(A,B)\cong\overline{\mathrm{Ext}}^1(A,J_B)=0$ by \cite[Lemma 9.1]{E} and  Lemma \ref{lem2.3}. Let $U\in{^\bot}\textsf{B}$, and let $Y\in\mathcal{B}_\textsf{I}$ and $Y\stackrel{\simeq}\rightarrow I$ be a semi-injective resolution with $I\in\mathrm{CE}\textrm{-}\textsf{B}$.
Then
$\overline{\mathrm{Ext}}^1(U[-s],I)\cong\mathrm{Ext}^1_R(U,\mathrm{Z}^s(I))=0$ by \cite[Lemma 9.1]{E}, so $\xi xt^{1}_\xi(U[-s],Y)=0$ by Lemma \ref{lem2.3}. Hence $U[-s]\in\mathcal{A}_\textsf{P}$, and so $U\in\textsf{A}$ by Lemma \ref{lem2.1}. Similarly, if $V\in\textsf{A}^\bot$, then $V\in\textsf{B}$. Thus $(\textsf{A},\textsf{B})$ is a cotorsion pair in $\mathrm{Mod}R$.
 Let $M\in\mathrm{Mod}R$. There are two cohomologically ghost triangles $Y\rightarrow X\rightarrow M\rightarrow Y[1]$
and
$M\rightarrow Y'\rightarrow X'\rightarrow M[1]$ in $\mathrm{D}(R)$ with $X,X'\in\mathcal{A}_\textsf{P}$ and $Y,Y'\in\mathcal{B}_\textsf{I}$, its yield two short exact sequences of $R$-modules $$0\rightarrow \mathrm{H}^0(Y)\rightarrow \mathrm{H}^0(X)\rightarrow M\rightarrow 0,\
0\rightarrow M\rightarrow \mathrm{H}^0(Y')\rightarrow \mathrm{H}^0(X')\rightarrow 0$$ with $\mathrm{H}^0(X),\mathrm{H}^0(X')\in\textsf{A}$ and $\mathrm{H}^0(Y),\mathrm{H}^0(Y')\in\textsf{B}$ by Lemmas \ref{lem2.1} and \ref{lem2.2}. The hereditariness is easy to check. Thus $(\textsf{A},\textsf{B})$ is a complete hereditary cotorsion pair in $\mathrm{Mod}R$.
\end{proof}

\begin{cor}\label{lem2.7} $(\textsf{A},\textsf{C})$ and $(\textsf{C},\textsf{B})$ are complete hereditary cotorsion pairs in $\mathrm{Mod}R$
 if and only if $(\mathcal{A}_\textsf{P},\mathcal{C}_\textsf{I})$ and ($\mathcal{C}_\textsf{I},\mathcal{B}_\textsf{I})$ are complete hereditary cotorsion pairs in $\mathrm{D}(R)$ with respect to $\xi$.
\end{cor}
\begin{proof} It suffices to prove that $\mathcal{C}_\textsf{P}=\mathcal{C}_\textsf{I}$. Let $Z\in\mathcal{C}_\textsf{I}$ and $Y\in\mathcal{B}_\textsf{I}$, and let $Z\stackrel{\simeq}\rightarrow I$ and $Y\stackrel{\simeq}\rightarrow J$ be semi-injective resolutions, respectively. As $\textsf{A}^\bot=\textsf{C}={^\bot}\textsf{B}$, $\overline{\mathrm{Ext}}^1(I,J)=0$ by \cite[Theorem 9.4]{E}, it follows by Lemma \ref{lem2.3} that
$\xi xt^{1}_\xi(Z,Y)=0$, so $Z\in\mathcal{C}_\textsf{P}$. By dual arguments and Theorem \ref{lem2.5}, one has $\mathcal{C}_\textsf{I}=\mathcal{C}_\textsf{P}$. The converse follows by Theorem \ref{lem2.5}.
\end{proof}

\bigskip
\section{\bf Some examples of complete hereditary cotorsion pairs in $\mathrm{D}(R)$}
In this section, some examples of complete hereditary cotorsion pairs with respect to $\xi$ in  $\mathrm{D}(R)$ are given. As  consequences,
\cite[Theorem 1.1]{HZZ} can be applied to yield numerous model structures on $(\mathrm{D}(R,[1],\xi)$.

\subsection{\bf $\xi$-$\mathcal{G}$projective and $\xi$-$\mathcal{G}$injective objects in $\mathrm{D}(R)$}

Following \cite{EJ95}, an $R$-module $M$ is called \emph{Gorenstein projective} if there
exists a $\textrm{Hom}_R(-,\textsf{P})$-exact exact complex \begin{center}$P:\ \cdots\longrightarrow P^{-1}\longrightarrow P^0\longrightarrow P^1\longrightarrow\cdots$\end{center}of projective $R$-modules
such that $M\cong\textrm{Z}^0(P)$. The class of Gorenstein projective
$R$-modules is denoted by $\textsf{GP}$.
 Gorenstein injective
module is defined dually. The class of Gorenstein injective modules is denoted by
$\textsf{GI}$.

\begin{df}\label{lem:1.2} {\rm (\cite{AS}). A complex $X$ is called \emph{$\xi$-$\mathcal{G}$projective} (resp. \emph{$\xi$-$\mathcal{G}$injective}) if for each $t$, there is a cohomologically ghost triangle $X_{t+1}\rightarrow P_t\rightarrow X_t\stackrel{\nu_t}\rightarrow X_{t+1}[1]$ with $P_t\in\langle\textsf{P}\rangle_1$ (resp. $P_t\in\langle\textsf{I}\rangle_1$), such that
\begin{align}
0\rightarrow\mathrm{Hom}_{\mathrm{D}(R)}(X_{t},P)\rightarrow\mathrm{Hom}_{\mathrm{D}(R)}(P_t,P)
\rightarrow\mathrm{Hom}_{\mathrm{D}(R)}(X_{t+1},P)\rightarrow0
\label{exact03}
\tag{$\ast$}\end{align}
\begin{align}
(\textrm{resp.}\ 0\rightarrow\mathrm{Hom}_{\mathrm{D}(R)}(I,X_{t+1})\rightarrow\mathrm{Hom}_{\mathrm{D}(R)}(I,P_t)
\rightarrow\mathrm{Hom}_{\mathrm{D}(R)}(I,X_{t})\rightarrow0)
\label{exact03}
\tag{$\ast'$}\end{align}is exact in Ab  for all $P\in\langle\textsf{P}\rangle_1$ (resp. $I\in\langle\textsf{I}\rangle_1$) and $X\simeq X_0$. The symbol $\mathcal{GP}(\xi)$ (resp. $\mathcal{GI}(\xi)$) denote the full subcategory of $\xi$-$\mathcal{G}$projective (resp. $\xi$-$\mathcal{G}$injective) objects of $\mathrm{D}(R)$.

As $\mathrm{Hom}_{\mathrm{D}(R)}(-,-[s])\cong\mathrm{H}^s(\mathrm{RHom}_{R}(-,-))$ for $s\in\mathbb{Z}$, $(\ast)$ (resp. $(\ast')$) is exact if and only if $\mathrm{RHom}_R(\nu_t,P)$ (resp. $\mathrm{RHom}_R(I,\nu_t)$) is a ghost for all $P\in\langle\textsf{P}\rangle_1$ (resp. $I\in\langle\textsf{I}\rangle_1$).}
\end{df}

For two complexes $X$ and $Y$, Christensen, Foxby and Holm \cite[Proposition 2.5.8]{CFH} established a morphism
$$\mathrm{H}(\mathrm{Hom}_R(X,Y))\rightarrow\mathrm{Hom}_R(\mathrm{H}(X),\mathrm{H}(Y))$$ with $\eta^s_{X,Y}:\mathrm{H}^s(\mathrm{Hom}_R(X,Y))\rightarrow\prod_{i\in\mathbb{Z}}\mathrm{Hom}_R(\mathrm{H}^i(X),\mathrm{H}^{i+s}(Y))$ given by $[\alpha]\mapsto(\mathrm{H}^i(\alpha))_{i\in\mathbb{Z}}$. Let $T\stackrel{\simeq}\rightarrow X$ be a semi-projective resolution. For $s\in\mathbb{Z}$, one has a commutative diagram:
\begin{center} $\xymatrix@C=30pt@R=16pt{
\mathrm{H}^{s}(\mathrm{RHom}_{R}(X,Y))\ar[d]^\cong\ar@{.>}[r]^{\rho^s_{X,Y}\ \ \ \ \ \ }&\prod_{i\in\mathbb{Z}}\mathrm{Hom}_R(\mathrm{H}^i(X),\mathrm{H}^{i+s}(Y))\ar[d]^\cong\\
 \mathrm{H}^{s}(\mathrm{Hom}_{R}(T,Y))\ar[r]^{\eta^s_{X,Y}\ \ \ \ \ \ \ }&\prod_{i\in\mathbb{Z}}\mathrm{Hom}_R(\mathrm{H}^i(T),\mathrm{H}^{i+s}(Y))}$
\end{center}

For any ring $R$, $(\textsf{GP},\textsf{GP}^{\bot})$ is a hereditary cotorsion pair in $\mathrm{Mod}R$ by \cite[Corollary 3.4(1)]{CIS}. The following proposition and lemma \ref{lem2.1} show that $\mathcal{GP}(\xi)=\mathcal{GP}_\textsf{P}$.

\begin{prop}\label{lem1.3} Let $X\in\mathrm{D}(R)$. The following statements are equivalent:

$(1)$ $X$ is $\xi$-$\mathcal{G}$projective in $\mathrm{D}(R)$;

$(2)$ $\mathrm{H}^{-s}(X)\in\textsf{GP}$ and $\mathrm{H}^{s}(\mathrm{RHom}_{R}(X,B))\stackrel{\cong}\rightarrow\mathrm{Hom}_{R}(\mathrm{H}^{-s}(X),B)$ for all $B\in\textsf{GP}^\perp$ and $s\in\mathbb{Z}$.
\end{prop}
\begin{proof} (1) $\Rightarrow$ (2) By assumption, one has a cohomologically ghost triangle $X_1\rightarrow P_0\rightarrow X\stackrel{\nu}\rightarrow X_1[1]$ in $\mathrm{D}(R)$ with $P_0\in\langle\textsf{P}\rangle_1$ and $X_1\in\mathcal{GP}(\xi)$, such that $\mathrm{RHom}_{R}(\nu,Q)$ is a ghost for $Q\in\langle\textsf{P}\rangle_1$. For $B\in\textsf{GP}^\perp$ and $s\in\mathbb{Z}$, we have a commutative diagram:
\begin{center} $\xymatrix@C=16pt@R=16pt{
 0\ar[r]& \mathrm{H}^s(\mathrm{RHom}_{R}(X,B)) \ar[d]^{\rho^{s}_{X,B}}\ar[r]& \mathrm{H}^s(\mathrm{RHom}_{R}(P_0,B))\ar[d]^{\rho^{s}_{P_0,B}}\ar[r] &\mathrm{H}^s(\mathrm{RHom}_{R}(X_1,B))\ar[d]^{\rho^{s}_{X_1,B}}\ar[r]&0 \\
0\ar[r]& \mathrm{Hom}_{R}(\mathrm{H}^{-s}(X),B)\ar[r]& \mathrm{Hom}_{R}(\mathrm{H}^{-s}(P_0),B)\ar[r] &\mathrm{Hom}_{R}(\mathrm{H}^{-s}(X_1),B)&}$
\end{center}By Lemma \ref{lem4.1}(1.1), $\rho^{s}_{P_0,B}$ is an isomorphism, so is $\rho^{s}_{X,B}$. Also
for $t\in\mathbb{Z}$, there is a cohomologically ghost triangle $X_{t+1}\rightarrow P_t\rightarrow X_t\stackrel{\nu_t}\rightarrow X_{t+1}[1]$ in $\mathrm{D}(A)$ with $P_t\in\langle\textsf{P}\rangle_1$ and $X\simeq X_0$ , such that $\mathrm{Hom}_{\mathrm{D}(A)}(\nu_t,Q)=0$ for $Q\in\langle\textsf{P}\rangle_1$. So one has an exact sequence of projective $R$-modules $$\mathbb{P}:\cdots\rightarrow\mathrm{H}^{-s}(P_1)\rightarrow\mathrm{H}^{-s}(P_0)\rightarrow \mathrm{H}^{-s}(P_{-1})\rightarrow\cdots$$  such that $\mathrm{Hom}_R(\mathbb{P},P)$ is exact for all $P\in\textsf{P}$ by the above isomorphism. Hence $\mathrm{H}^{-s}(X)\cong\mathrm{Coker}(\mathrm{H}^{-s}(P_1)\rightarrow\mathrm{H}^{-s}(P_0))\in\textsf{GP}$ for all $s\in\mathbb{Z}$.

 (2) $\Rightarrow$ (1) As each $\mathrm{H}^{-s}(X)$ is Gorenstein projective, one has a $\mathrm{Hom}_R(-,\textsf{P})$-exact exact sequence of left $R$-modules $$0\rightarrow\mathrm{H}^{-s}(X)\stackrel{f^{-s}}\rightarrow P_{-1_{-s}}\rightarrow X_{-1_{-s}}\rightarrow0$$  with $P_{-1_{-s}}\in\textsf{P}$ and $X_{-1_{-s}}\in\textsf{GP}$. As $$\mathrm{Hom}_{\mathrm{D}(R)}(X,P_{-1_{-s}}[s])\cong\mathrm{H}^{s}(\mathrm{RHom}_{R}(X,P_{-1_{-s}}))\cong\mathrm{Hom}_{R}(\mathrm{H}^{-s}(X),P_{-1_{-s}}),$$ there exists $g^{-s}:X\rightarrow P_{-1_{-s}}[s]$ in $\mathrm{D}(R)$ such that $\mathrm{H}^{-s}(g^{-s})=f^{-s}$. Set $P_{-1}=\bigoplus_{s\in\mathbb{Z}}P_{-1_{-s}}[s]=\prod_{s\in\mathbb{Z}}P_{-1_{-s}}[s]$. One has a cohomologically ghost triangle $$X\rightarrow P_{-1}\rightarrow X_{-1}\stackrel{\nu_{-1}}\rightarrow X[1]$$ in $\mathrm{D}(R)$, so that $\mathrm{H}^{-s}(X_{-1})\cong X_{-1_{-s}}$. For $P\in\textsf{P}$,
we have a commutative diagram:
\begin{center} $\xymatrix@C=20pt@R=20pt{
 & \mathrm{H}^{s}(\mathrm{RHom}_{R}(X_{-1},P)) \ar[d]\ar[r]& \mathrm{H}^{s}(\mathrm{RHom}_{R}(P_{-1},P))\ar[d]^\cong\ar[r] &\mathrm{H}^{s}(\mathrm{RHom}_{R}(X,P))\ar[d]^\cong \\
0\ar[r]& \mathrm{Hom}_{R}(\mathrm{H}^{-s}(X_{-1}),P)\ar[r]& \mathrm{Hom}_{R}(\mathrm{H}^{-s}(P_{-1}),P)\ar[r] &\mathrm{Hom}_{R}(\mathrm{H}^{-s}(X),P)\ar[r]&0 }$
\end{center}it implies that the upper row is exact and $\mathrm{H}^{s}(\mathrm{RHom}_{R}(X_{-1},P))\cong\mathrm{Hom}_{R}(\mathrm{H}^{-s}(X_{-1}),P)$. On the other hand, consider the cohomologically ghost triangle $X_1\rightarrow P_{0}\rightarrow X\stackrel{\nu_0}\rightarrow X[1]$ in $\mathrm{D}(R)$ with $P_{0}\in\langle\textsf{P}\rangle_1$. For $P\in\textsf{P}$ and $s\in\mathbb{Z}$, we have a commutative diagram:\begin{center} $\xymatrix@C=20pt@R=20pt{
 & \mathrm{H}^{s}(\mathrm{RHom}_{R}(X,P)) \ar[d]^\cong\ar[r]& \mathrm{H}^{s}(\mathrm{RHom}_{R}(P_{0},P))\ar[d]^\cong\ar[r] &\mathrm{H}^{s}(\mathrm{RHom}_{R}(X_1,P))\ar[d] \\
0\ar[r]& \mathrm{Hom}_{R}(\mathrm{H}^{-s}(X),P)\ar[r]& \mathrm{Hom}_{R}(\mathrm{H}^{-s}(P_{0}),P)\ar[r] &\mathrm{Hom}_{R}(\mathrm{H}^{-s}(X_1),P)\ar[r]&0 }$
\end{center}Then $\mathrm{RHom}_R(\nu_0,P)$ is a ghost, it implies that $\mathrm{H}^{s}(\mathrm{RHom}_{R}(X_1,P))\cong\mathrm{Hom}_{R}(\mathrm{H}^{-s}(X_1),P)$.
Continuing this process, one has $X$ is $\xi$-$\mathcal{G}$projective in $\mathrm{D}(R)$.
\end{proof}

For any ring $R$, $({^\bot}\textsf{GI},\textsf{GI})$ is a complete hereditary cotorsion pair in $\mathrm{Mod}R$ by \cite[Theorem 5.6]{SS}. By dual arguments, one can prove that $\mathcal{GI}(\xi)=\mathcal{GI}_\textsf{I}$.

\begin{prop}\label{lem1.4} Let $Y\in\mathrm{D}(R)$. The following statements are equivalent:

$(1)$ $Y$ is $\xi$-$\mathcal{G}$injective in $\mathrm{D}(R)$;

$(2)$ $\mathrm{H}^{s}(Y)\in\textsf{GI}$ and $\mathrm{H}^{s}(\mathrm{RHom}_{R}(A,Y))\stackrel{\cong}\rightarrow\mathrm{Hom}_{R}(A,\mathrm{H}^{s}(Y))$ for all $A\in{^\perp}\textsf{GI}$ and $s\in\mathbb{Z}$.
\end{prop}

\subsection{\bf Projectively coresolved $\xi$-$\mathcal{G}$flat objects in $\mathrm{D}(R)$}

Following \cite{SS}, an $R$-module $M$ is called \emph{projectively coresolved Gorenstein flat} if there
exists an $\textsf{I}\otimes_R-$-exact exact complex \begin{center}$P:\ \cdots\longrightarrow P^{-1}\longrightarrow P^0\longrightarrow P^1\longrightarrow\cdots$\end{center}of projective $R$-modules
such that $M\cong\textrm{Z}^0(P)$. The class of projectively coresolved Gorenstein flat
left $R$-modules is denoted by $\textsf{PGF}$.

\begin{df}\label{df:3.1} {\rm A complex $X$ in $\mathrm{D}(R)$ is called \emph{projectively coresolved $\xi$-$\mathcal{G}$flat} if for each $t$, there is a cohomologically ghost triangle $X_{t+1}\rightarrow P_t\rightarrow X_t\stackrel{\nu_t}\rightarrow X_{t+1}[1]$ in $\mathrm{D}(R)$ with $P_t\in\langle\textsf{P}\rangle_1$, such that $X\simeq X_0$ and $$I\otimes_R^\mathrm{L}X_{t+1}\longrightarrow I\otimes_R^\mathrm{L}P_t\longrightarrow I\otimes_R^\mathrm{L}X_t\stackrel{I\otimes^\mathrm{L}_R\nu_t}\longrightarrow I\otimes_R^\mathrm{L}X_{t+1}[1]$$ is a cohomologically ghost triangle in Ab for all $I\in\langle\textsf{I}\rangle_1$.  The symbol $\mathcal{PGF}(\xi)$  denote the full subcategory of projectively coresolved $\xi$-$\mathcal{G}$flat objects of $\mathrm{D}(R)$.}
\end{df}

For any ring $R$, $(\textsf{PGF},\textsf{PGF}^\bot)$ is a complete hereditary cotorsion pair in $\mathrm{Mod}R$ by \cite[Theorem 4.9]{SS}. The following lemma shows that $\mathcal{PGF}(\xi)=\mathcal{PGF}_\textsf{P}$.

\begin{prop}\label{lem1.1} Let $X\in\mathrm{D}(R)$. The following statements are equivalent:

$(1)$ $X$ is projectively coresolved $\xi$-$\mathcal{G}$flat in $\mathrm{D}(R)$;

$(2)$ For any semi-projective resolution $T\stackrel{\simeq}\rightarrow X$, $\mathrm{C}^{s}(T)\in\textsf{PGF}$ for all $s\in\mathbb{Z}$.
\end{prop}
\begin{proof} (1) $\Rightarrow$ (2) For each $t$, there is a cohomologically ghost triangle\begin{align}
X_{t+1}\longrightarrow P_t\longrightarrow X_t\stackrel{\nu_t}\longrightarrow X_{t+1}[1]
\label{exact04}\tag{$\dag_t$}\end{align}in $\mathrm{D}(R)$ with $P_t\in\langle\textsf{P}\rangle_1$ and $X\simeq X_0$, such that $I\otimes^\mathrm{L}_R\nu_t$ is a ghost for all $I\in\langle\textsf{I}\rangle_1$. Let $f_0:T_0\stackrel{\simeq}\rightarrow X_0$ be a semi-projective resolution, and set $T_1=\mathrm{cn}(P_0\rightarrow T_0)[-1]$ and $T_{-1}=\mathrm{cn}(T_0\rightarrow P_{-1})$. One has two commutative diagrams of exact triangles
 \begin{center} $\xymatrix@C=18pt@R16pt{
T_1\ar[d]^{f_1}\ar[r]& P_0\ar@{=}[d]\ar[r] &T_0\ar[d]^{f_0}_\simeq\ar[r]^{\mu_0\ }&T_1[1]\ar[d] \\
X_1 \ar[r]& P_0\ar[r] &X_0\ar[r]^{\nu_0\ }&X_1[1] }$\quad $\xymatrix@C=18pt@R16pt{
T_0\ar[d]^{f_0}_\simeq\ar[r]& P_{-1}\ar@{=}[d]\ar[r] &T_{-1}\ar[d]^{f_{-1}}\ar[r]^{\mu_{-1}\ }&T_0[1]\ar[d] \\
X_0 \ar[r]& P_{-1}\ar[r] &X_{-1}\ar[r]^{\nu_{-1}\ }&X_0[1] }$
\end{center}Then $f_1$ and $f_{-1}$ are quasi-isomorphisms and $\mu_0,\mu_{-1}$ and $I\otimes_R\mu_0,I\otimes_R\mu_{-1}$  are ghosts for all $I\in\langle\textsf{I}\rangle_1$. By repeating this process on $T_1\stackrel{\simeq}\rightarrow X_1$ and $T_{-1}\stackrel{\simeq}\rightarrow X_{-1}$,  for each $t$, one has a cohomologically ghost triangle\begin{align}
T_{t+1}\longrightarrow P_t\longrightarrow T_t\stackrel{\mu_t}\longrightarrow T_{t+1}[1]
\label{exact04}\tag{$\ddag_t$}\end{align}in $\mathrm{K}(\textsf{P})$, such that $I\otimes_R\mu_t$ is a ghost for all $I\in\langle\textsf{I}\rangle_1$.
 For each $T_t$, one has a CE exact sequence $0\rightarrow T'_{t+1}\rightarrow P_t\oplus P'_t\rightarrow T_t\rightarrow0$ in $\mathrm{C}(\textsf{P})$, where $P'_t$ is a projective complex. Notice $T'_t\rightarrow T_t$ is a homotopy equivalence, one has an epimorphism $T'_t\oplus\bar{P'}_t\rightarrow T_t$ for some projective complex $\bar{P'}_t$, such that $\bar{P}_t:=\mathrm{Ker}(T'_t\oplus\bar{P'}_t\rightarrow T_t)\in\mathrm{K}(\textsf{P})$ is contractible, so $T_t\oplus \bar{P}_t\cong T'_t\oplus\bar{P'}_t$ in $\mathrm{C}(R)$ and $\bar{P}_t$ is a projective complex. Thus we have a degreewise split CE exact sequence
$$0\rightarrow T_{t+1}\oplus\bar{P}_{t+1}\rightarrow P_t\oplus P'_t\oplus \bar{P}_t\oplus\bar{P'}_{t+1}\rightarrow T_t\oplus \bar{P}_t\rightarrow0,$$
 so that the induced triangle is isomorphic to $(\ddag_t)$ in $\mathrm{K}(\textsf{P})$, it implies that $$0\rightarrow\mathrm{H}^s(I\otimes_R(T_t\oplus \bar{P}_t))\rightarrow\mathrm{H}^s(I\otimes_R(P_t\oplus P'_t\oplus \bar{P}_t\oplus\bar{P'}_{t+1}))\rightarrow\mathrm{H}^s(I\otimes_R(T_{t+1}\oplus\bar{P}_{t+1}))\rightarrow0$$ is exact
for all $I\in\textsf{I}$. As $0\rightarrow I\otimes_R(T_{t+1}\oplus\bar{P}_{t+1})\rightarrow I\otimes_R(P_t\oplus P'_t\oplus \bar{P}_t\oplus\bar{P'}_{t+1})\rightarrow I\otimes_R(T_t\oplus \bar{P}_t)\rightarrow0$ is exact, it follows by
  Lemma \ref{lem0.0} that this sequence is CE exact, so the sequence $0\rightarrow \mathrm{Hom}_{R}(T_t\oplus \bar{P}_t,I^+)\rightarrow \mathrm{Hom}_{R}(P_t\oplus P'_t\oplus \bar{P}_t\oplus\bar{P'}_{t+1},I^+)\rightarrow \mathrm{Hom}_{R}(T_{t+1}\oplus\bar{P}_{t+1},I^+)\rightarrow0$ is CE exact.
For $s\in\mathbb{Z}$, by \cite[Lemma 3.1]{G}, $$\mathrm{Z}^{-s}((I\otimes_R-)^+)\cong\mathrm{Z}^{-s}(\mathrm{Hom}_{R}(-,I^+))\cong\mathrm{Hom}_{\mathrm{C}(R)}(-,I^+[-s])\cong\mathrm{Hom}_{R}(\mathrm{C}^s(-),I^+),$$
one has $I\otimes_R\mathrm{C}^s(-)\cong\mathrm{C}^{s}(I\otimes_R-)$. Also $\mathrm{C}^s(P_t\oplus P'_t\oplus \bar{P}_t\oplus\bar{P'}_{t+1})\in\textsf{P}$, it follows that
$\mathrm{C}^s(T_0)\oplus\mathrm{C}^s(\bar{P}_0)\cong\mathrm{C}^s(T_0\oplus \bar{P}_0)\in\textsf{PGF}$, so $\mathrm{C}^{s}(T_0)\in\textsf{PGF}$ for all $s\in\mathbb{Z}$.

(2) $\Rightarrow$ (1) As $\mathrm{C}^{s}(T)\in\textsf{PGF}$ for all $s\in\mathbb{Z}$, one has a $\textsf{I}\otimes_R-$-exact exact sequence $$0\rightarrow\mathrm{C}^{s}(T)\stackrel{f^{s}}\rightarrow P_{-1_{s}}\rightarrow T_{-1_{s}}\rightarrow0$$  of $R$-modules with $P_{-1_{s}}\in\textsf{P}$ and $T_{-1_{s}}\in\textsf{PGF}$.
By \cite[Lemma 3.1]{G}, one has  $$\mathrm{Hom}_{R}(\mathrm{C}^{s}(T),P_{-1_{s}})\cong\mathrm{Hom}_{\mathrm{C}(R)}(T,P_{-1_{s}}[-s]),$$there is a chain map  $g^{s}:T\rightarrow P_{-1_{s}}[-s]$ such that $g^{s}d^{s-1}_T=0$. Set $P_{-1}=\bigoplus_{s\in\mathbb{Z}}P_{-1_{s}}[-s]\cong\prod_{s\in\mathbb{Z}}P_{-1_{s}}[-s]$. Then there is a chain map  $g:T\rightarrow P_{-1}$ such that $\pi^{s}g=g^s$, where $\pi^s:P_{-1}\rightarrow P_{-1_{s}}[-s]$ is the projection. We have a commutative diagram:
\begin{center} $\xymatrix@C=18pt@R=16pt{
 0\ar[r]& T \ar[d]\ar[r]& \mathrm{cn}(\mathrm{id}_T)\ar[d]\ar[r] &T[1]\ar@{=}[d]\ar[r]&0 \\
0\ar[r]& P_{-1}\ar[r]& T_{-1}\ar[r] &T[1]\ar[r]&0 }$
\end{center}it induced two commutative diagrams:\begin{center} $\xymatrix@C=14pt@R=16pt{
     &  0\ar[d]  &  &  \\
       &\mathrm{C}^{s}(T)\ar[r] \ar[d]& T^{s+1}\ar[d] \ar[r]& \mathrm{C}^{s+1}(T)\ar@{=}[d]\ar[r]&0\\
     & P_{-1_{s}} \ar[d] \ar[r] &\mathrm{C}^{s}(T_{-1})  \ar[d]\ar[r] &\mathrm{C}^{s+1}(T)  \ar[r] & 0  \\
     &    T_{-1_{s}} \ar[d]\ar@{=}[r] & T_{-1_{s}}\ar[d]  \\
    &  0& 0 &}$\quad  $\xymatrix@C=14pt@R=16pt{
      & & 0\ar[d]  & 0 \ar[d] &  \\
    &&T^{s+1}\ar[d] \ar@{=}[r] & T^{s+1} \ar[d] \\
 0\ar[r]& \mathrm{C}^{s}(T) \ar@{=}[d]\ar[r]& T^{s+1}\oplus P^s_{-1}\ar[d]\ar[r] &\mathrm{C}^{s}(T_{-1})\ar[d]\ar[r]&0 \\
 0\ar[r]& \mathrm{C}^{s}(T) \ar[r]& P_{-1_{s}}\ar[d]\ar[r] &T_{-1_{s}}\ar[d]\ar[r]&0 \\
    & & 0& 0  &}$
\end{center}Note that $T^{s+1}\oplus P^s_{-1}=\mathrm{C}^{s}(\mathrm{cn}(\mathrm{id}_T)\oplus P_{-1})$,
the sequence $0\rightarrow T\rightarrow\mathrm{cn}(\mathrm{id}_T)\oplus P_{-1}\rightarrow T_{-1}\rightarrow0$ in $\mathrm{C}(\textsf{P})$ is CE exact by Lemma \ref{lem0.0}, so that the exact sequence $0\rightarrow\mathrm{C}^{s}(T)\rightarrow T^{s+1}\oplus P^s_{-1}\rightarrow\mathrm{C}^{s}(T_{-1})\rightarrow0$ is $\textsf{I}\otimes_R-$-exact for all $s\in\mathbb{Z}$. Thus one has a cohomologically ghost triangle $$X\rightarrow P_{-1}\rightarrow X_{-1}\stackrel{\nu_{-1}}\rightarrow X[1]$$ in $\mathrm{D}(R)$, such that $T_{-1}\simeq X_{-1}$ is a semi-projective resolution and $\mathrm{C}^{s}(T_{-1})\in\textsf{PGF}$. On the other hand, there exists a CE exact sequence $0\rightarrow T_1\rightarrow P_{0}\rightarrow T\rightarrow 0$ in $\mathrm{C}(\textsf{P})$ with $P_0\in\mathrm{CE}\textrm{-}\textsf{P}$, it induces a cohomologically ghost triangle $$X_1\rightarrow P_0\rightarrow X_0\stackrel{\nu_0}\rightarrow X_1[1]$$ in $\mathrm{D}(R)$ with $P_0\in\langle\textsf{P}\rangle_1$ such that $T_{1}\simeq X_{1}$ is a semi-projective resolution and $\mathrm{C}^{s}(T_1)\in\textsf{PGF}$ for all $s\in\mathbb{Z}$. By repeating this process on $X_{-1}$ and $X_1$, one has $X\in\mathcal{PGF}(\xi)$.
\end{proof}

\subsection{\bf $\xi$-$\mathcal{G}$flat objects in $\mathrm{D}(R)$}

Following \cite{EJT}, an $R$-module $M$ is called \emph{Gorenstein flat} if there
is an $\textsf{I}\otimes_R-$-exact exact complex \begin{center}$F:\ \cdots\longrightarrow F^{-1}\longrightarrow F^0\longrightarrow F^1\longrightarrow\cdots$\end{center}of flat $R$-modules
such that $M=\textrm{Z}^0(F)$. The class of  Gorenstein flat
left $R$-modules is denoted by $\textsf{GF}$.

\begin{df}\label{df:3.1} {\rm A complex $X$ in $\mathrm{D}(R)$ is called \emph{$\xi$-$\mathcal{G}$flat} if for each $t$, there is a cohomologically ghost triangle $X_{t+1}\rightarrow F_t\rightarrow X_t\stackrel{\nu_t}\rightarrow X_{t+1}[1]$ in $\mathrm{D}(R)$ with $F_t\in\langle\textsf{F}\rangle_1$, such that $$I\otimes_R^\mathrm{L}X_{t+1}\longrightarrow I\otimes_R^\mathrm{L}F_t\longrightarrow I\otimes_R^\mathrm{L}X_t\stackrel{I\otimes^\mathrm{L}_R\nu_t}\longrightarrow I\otimes_R^\mathrm{L}X_{t+1}[1]$$ is a cohomologically ghost triangle in Ab for all $I\in\langle\textsf{I}\rangle_1$ and $X\simeq X_0$.  The symbol $\mathcal{GF}(\xi)$  denote the full subcategory of $\xi$-$\mathcal{G}$flat objects of $\mathrm{D}(R)$.}
\end{df}

For any ring $R$, $(\textsf{GF},\textsf{GF}^\bot)$ is a complete hereditary cotorsion pair in $\mathrm{Mod}R$ by \cite[Corollary 4.12]{SS}. The next result proves that  $\mathcal{GF}(\xi)=\mathcal{GF}_\textsf{P}$.

\begin{prop}\label{lem1.0} Let $X\in\mathrm{D}(R)$. The following statements are equivalent:

$(1)$ $X$ is $\xi$-$\mathcal{G}$flat in $\mathrm{D}(R)$;

$(2)$ For any semi-projective resolution $T\stackrel{\simeq}\rightarrow X$, $\mathrm{C}^{s}(T)\in\textsf{GF}$ for all $s\in\mathbb{Z}$.
\end{prop}
\begin{proof} (1) $\Rightarrow$ (2) For each $t$, there is a cohomologically ghost triangle\begin{align}
X_{t+1}\longrightarrow F_t\longrightarrow X_t\stackrel{\nu_t}\longrightarrow X_{t+1}[1]
\label{exact04}\tag{$\dag_t$}\end{align}in $\mathrm{D}(R)$ with $F_t\in\langle\textsf{F}\rangle_1$ and $X\simeq X_0$, such that $I\otimes^\mathrm{L}_R\nu_t$ is a ghost for all $I\in\langle\textsf{I}\rangle_1$. Let $f_0:T_0\stackrel{\simeq}\rightarrow X_0$, $g_0:P_0\stackrel{\simeq}\rightarrow F_0$ and $g_{-1}:P_{-1}\stackrel{\simeq}\rightarrow F_{-1}$ be semi-projective resolutions, and set $T_1=\mathrm{cn}(P_0\rightarrow T_0)[-1]$ and $T_{-1}=\mathrm{cn}(T_0\rightarrow P_{-1})$. One has two commutative diagrams of exact triangles
 \begin{center} $\xymatrix@C=18pt@R16pt{
T_1\ar[d]^{f_1}\ar[r]& P_0\ar[d]^{g_0}_\simeq\ar[r] &T_0\ar[d]^{f_0}_\simeq\ar[r]^{\mu_0\ }&T_1[1]\ar[d] \\
X_1 \ar[r]& F_0\ar[r] &X_0\ar[r]^{\nu_0\ }&X_1[1] }$\quad $\xymatrix@C=18pt@R16pt{
T_0\ar[d]^{f_0}_\simeq\ar[r]& P_{-1}\ar[d]^{g_{-1}}_\simeq\ar[r] &T_{-1}\ar[d]^{f_{-1}}\ar[r]^{\mu_{-1}\ }&T_0[1]\ar[d] \\
X_0 \ar[r]& F_{-1}\ar[r] &X_{-1}\ar[r]^{\nu_{-1}\ }&X_0[1] }$
\end{center}Then $f_1$ and $f_{-1}$ are quasi-isomorphisms and $\mu_0,\mu_{-1}$ and $I\otimes_R\mu_0,I\otimes_R\mu_{-1}$  are ghosts for all $I\in\langle\textsf{I}\rangle_1$. By repeating this process on $T_1\stackrel{\simeq}\rightarrow X_1$ and $T_{-1}\stackrel{\simeq}\rightarrow X_{-1}$,  for each $t$, one has a cohomologically ghost triangle\begin{align}
T_{t+1}\longrightarrow P_t\longrightarrow T_t\stackrel{\mu_t}\longrightarrow T_{t+1}[1]
\label{exact04}\tag{$\ddag_t$}\end{align}in $\mathrm{K}(\textsf{P})$, such that $I\otimes_R\mu_t$ is a ghost for all $I\in\langle\textsf{I}\rangle_1$.
 Now, by analogy with the proof of Proposition \ref{lem1.1}, one has $\mathrm{C}^{s}(T_0)\in\textsf{GF}$ for all $s\in\mathbb{Z}$.

(2) $\Rightarrow$ (1) By analogy with the proof of Proposition \ref{lem1.1}.
\end{proof}

 A ring $R$ is called \emph{virtually Gorenstein} if $\textsf{GP}^\bot={^\bot}\textsf{GI}$. By \cite[Proposition 3.16]{WE}, $(\textsf{GP},\textsf{GP}^{\bot})$ is a complete hereditary cotorsion pair in $\mathrm{Mod}R$ whenever $R$ is virtually Gorenstein.
 The next corollary is an immediate consequence of Theorem \ref{lem2.5} and the preceding propositions.

\begin{cor}\label{lem1.5} $(1)$ If $R$ is a virtually Gorenstein ring, then $(\mathcal{GP}(\xi),\mathcal{GP}(\xi)^{\bot_\xi})$ is a complete hereditary cotorsion pair in $\mathrm{D}(R)$ with respect to $\xi$.

$(2)$ $({^{\bot_\xi}}\mathcal{GI}(\xi),\mathcal{GI}(\xi))$ is a complete hereditary cotorsion pair in $\mathrm{D}(R)$ with respect to $\xi$.

$(3)$ $(\mathcal{PGF}(\xi),\mathcal{PGF}(\xi)^{\bot_\xi})$ is a complete hereditary cotorsion pair in $\mathrm{D}(R)$ with respect to $\xi$.

$(4)$ $(\mathcal{GF}(\xi),\mathcal{GF}(\xi)^{\bot_\xi})$ is a complete hereditary cotorsion pair in $\mathrm{D}(R)$ with respect to $\xi$.

$(5)$ The ring $R$ is virtually Gorenstein if and only if $\mathcal{GP}(\xi)^{\bot_\xi}={^{\bot_\xi}}\mathcal{GI}(\xi)$.
\end{cor}

\subsection{\bf Cotorsion pairs arising from complexes of finite dimensions}

 For $n\geq0$, let  $\textsf{P}_{n}$ (resp. $\textsf{F}_{n}$,  $\textsf{I}_{n}$) be the class of $R$-modules of projective
(resp. flat, injective) dimension $\leq n$.
  An $R$-module $M$ has \emph{Gorenstein projective (resp. projectively coresolved Gorenstein flat, Gorenstein flat, Gorenstein injective) dimension at most $n$}, $\mathrm{Gpd}_RM\leq n$ (resp. $\mathrm{PGFd}_RM\leq n$, $\mathrm{Gfd}_RM\leq n$, $\mathrm{Gid}_RM\leq n$),
 if $M$ has a Gorenstein projective (resp. projectively coresolved Gorenstein flat, Gorenstein flat, Gorenstein injective) resolution of length $n$. Let $\textsf{GP}_n$ (resp. $\textsf{PGF}_n$, $\textsf{GF}_n$, $\textsf{GI}_n$)
be the class of $R$-modules $M$ of $\mathrm{Gpd}_RM\leq n$ (resp. $\mathrm{PGFd}_RM\leq n$, $\mathrm{Gfd}_RM\leq n$, $\mathrm{Gid}_RM\leq n$).

\begin{fact}\label{lem:2.2}{\rm For $n\geq0$, the following cotorsion pairs in $\mathrm{Mod}R$ is complete and hereditary:

(1) $(\textsf{P}_{n},\textsf{P}_n^\perp)$ by \cite[Theorem 7.4.6]{EJ}.

(2) $(\textsf{F}_{n},\textsf{F}_n^\perp)$ by \cite[Theorem 4.1.3]{GT}.

 (3) $({^\perp}\textsf{I}_{n},\textsf{I}_n)$ by \cite[Theorem 4.1.7]{GT}.

 (4) $(\textsf{PGF}_n,\textsf{P}_n^\perp\cap\textsf{PGF}^\bot)$ by \cite[Corollary 2.6]{GLZ}.

 (5) $(\textsf{GF}_n,\textsf{F}_n^\perp\cap\textsf{PGF}^\bot)$ by \cite[Theorem A]{M}.

 (6) $({^\perp}\textsf{I}_n\cap{^\perp}\textsf{GI},\textsf{GI}_n)$ by (3), \cite[Theorem 5.6]{SS} and the dual arguments of \cite[Theorem 3.4]{GLZ}.

 (7) $(\textsf{GP}_n,\textsf{P}_n^\perp\cap\textsf{GP}^\bot)$ whenever $R$  is virtually Gorenstein by \cite[Proposition 3.16]{WE} and \cite[Corollary 2.6]{GLZ}.}
\end{fact}

Following \cite{B}, we define inductively the \emph{$\xi$-$\mathcal{G}$projective dimension} $\xi\textrm{-Gpd}X$ of a complex
$X$ in $\mathrm{D}(R)$. When $X=0$, put $\xi\textrm{-Gpd}X=-1$. If $X\in\mathcal{GP}(\xi)$ then define $\xi\textrm{-Gpd}X=0$. For $n\geq1$,
define $\xi\textrm{-Gpd}X\leq n$ if there is  a cohomologically ghost triangle $K\rightarrow G\rightarrow X\rightarrow K[1]$, such that $G\in\mathcal{GP}(\xi)$ and $\xi\textrm{-Gpd}K\leq n-1$. This definition
is quite
different from those defined for cohomologically bounded complexes by Foxby \cite{F} and Yassemi \cite{Ya} and those
defined for unbounded complexes by Veliche \cite{V}.

Similarly, one can define $\xi$-projective dimension $\xi\textrm{-pd}X$, $\xi$-flat dimension $\xi\textrm{-fd}X$,  projectively coresolved $\xi$-$\mathcal{G}$flat dimension $\xi\textrm{-PGFd}X$, $\xi$-$\mathcal{G}$flat dimension $\xi\textrm{-Gfd}X$ of $X$. Let $\mathcal{P}(\xi)_n$ (resp. $\mathcal{F}(\xi)_n$, $\mathcal{GP}(\xi)_n$, $\mathcal{PGF}(\xi)_n$, $\mathcal{GF}(\xi)_n$)
be the class of complex $X$ of $\xi\textrm{-pd}X\leq n$ (resp. $\xi\textrm{-fd}X\leq n$, $\xi\textrm{-Gpd}X\leq n$, $\xi\textrm{-PGFd}X\leq n$, $\xi\textrm{-Gfd}X\leq n$).  Dually, one can define $\xi$-$\mathcal{G}$injective dimension $\xi\textrm{-Gid}X$ and $\xi$-injective dimension $\xi\textrm{-id}X$ of $X$.

\begin{prop}\label{lem3.1} Let $X\in\mathrm{D}(R)$ and $n\geq0$.

$(1)$ $X\in\mathcal{GP}(\xi)_n$ if and only if $X\in(\mathcal{GP}_{n})_\textsf{P}$.

$(2)$ $X\in\mathcal{PGF}(\xi)_n$ if and only if $X\in(\mathcal{PGF}_{n})_\textsf{P}$.

$(3)$ $X\in\mathcal{GF}(\xi)_n$ if and only if $X\in(\mathcal{GF}_{n})_\textsf{P}$.

$(4)$ $X\in\mathcal{F}(\xi)_n$ if and only if $X\in(\mathcal{F}_{n})_\textsf{P}$.

$(5)$ $X\in\mathcal{P}(\xi)_n$ if and only if $X\in(\mathcal{P}_{n})_\textsf{P}$.

$(6)$ $X\in\mathcal{GI}(\xi)_n$ if and only if $X\in(\mathcal{GI}_{n})_\textsf{I}$.

$(7)$ $X\in\mathcal{I}(\xi)_n$ if and only if $X\in(\mathcal{I}_{n})_\textsf{I}$.
\end{prop}
\begin{proof} We just prove one of the statements since the others are similar.

Let $X\in\mathcal{GP}(\xi)_n$. For $0\leq t\leq n-1$, there is a cohomologically ghost triangle\begin{align}
X_{t+1}\longrightarrow P_t\longrightarrow X_t\longrightarrow X_{t+1}[1]
\label{exact04}\tag{$\dag_t$}\end{align}in $\mathrm{D}(R)$ with $P_t\in\langle\textsf{P}\rangle_1$, $X_n\in\mathcal{GP}(\xi)$ and $X\simeq X_0$. By analogy with the proof of Proposition \ref{lem1.1}, there s a CE exact sequence $0\rightarrow T_{t+1}\rightarrow Q_t\rightarrow T_t\rightarrow0$ in $C(\textsf{P})$, such that the induced triangle is isomorphic to $\dag_t$, $Q_t\in\mathrm{CE}\textrm{-}\textsf{P}$ and $T_t$ is semi-projective. As $T_n\in\mathcal{GP}(\xi)$, it follows by Proposition \ref{lem1.3} and Lemma \ref{lem2.1} that $\mathrm{C}^s(T_n)\in\textsf{GP}$, so the exact sequence $$0\rightarrow\mathrm{C}^s(T_n)\rightarrow\mathrm{C}^s(Q_{n-1})\rightarrow\cdots\rightarrow\mathrm{C}^s(Q_0)\rightarrow\mathrm{C}^s(T_0)\rightarrow0$$ implies that $\mathrm{C}^s(T_0)\in\textsf{GP}_n$ for all $s\in\mathbb{Z}$. Thus $X\in(\mathcal{GP}_{n})_\textsf{P}$ by Lemma \ref{lem2.1}. Conversely, Let $F\stackrel{\simeq}\rightarrow X$ be a semi-projective resolution with $\mathrm{C}^s(F)\in\textsf{GP}_n$ for all $s\in\mathbb{Z}$.
Consider the CE exact sequence $\cdots\rightarrow Q_n\rightarrow Q_{n-1}\rightarrow\cdots\rightarrow Q_0\rightarrow F\rightarrow0$ with $Q_t\in\mathrm{CE}\textrm{-}\textsf{P}$ for $t\geq 0$, it induces the  cohomologically ghost triangle $F_{t+1}\rightarrow Q_t\rightarrow F_t\rightarrow F_{t+1}[1]$ where $F_t=\mathrm{coker}(Q_{t+1}\rightarrow Q_t)$ and $F\simeq F_0$. Note that $\mathrm{C}^s(F_n)\in\textsf{GP}$ for all $s\in\mathbb{Z}$, so $F_n\in\mathcal{GP}(\xi)$ by Lemma \ref{lem2.1} and Proposition \ref{lem1.3}. Hence $X\in\mathcal{GP}(\xi)_n$.
\end{proof}

\begin{cor}\label{lem3.2} Let $n$ be a non-negative integer.

$(1)$ $(\mathcal{P}(\xi)_n,\mathcal{P}(\xi)^{\bot_\xi}_n)$, $(\mathcal{F}(\xi)_n,\mathcal{F}(\xi)^{\bot_\xi}_n)$ and $({^{\bot_\xi}}\mathcal{I}(\xi)_n,\mathcal{I}(\xi)_n)$ are complete hereditary cotorsion pairs in $\mathrm{D}(R)$ with respect to $\xi$.

$(2)$ $(\mathcal{PGF}(\xi)_n,\mathcal{P}(\xi)^{\bot_\xi}_n\cap\mathcal{PGF}(\xi)^{\bot_\xi})$, $(\mathcal{GF}(\xi)_n,\mathcal{F}(\xi)^{\bot_\xi}_n\cap\mathcal{PGF}(\xi)^{\bot_\xi})$, $({^{\bot_\xi}}\mathcal{I}(\xi)_n\cap{^{\bot_\xi}}\mathcal{GI}(\xi),\mathcal{GI}(\xi)_n)$  are complete hereditary cotorsion pairs in $\mathrm{D}(R)$ with respect to $\xi$.

$(3)$  If $R$ is a virtually Gorenstein ring, then $(\mathcal{GP}(\xi)_n,\mathcal{P}(\xi)^{\bot_\xi}_n\cap\mathcal{GP}(\xi)^{\bot_\xi})$ is a complete hereditary cotorsion pair in $\mathrm{D}(R)$ with respect to $\xi$.
\end{cor}
\begin{proof} We just prove one of the statements since the other is similar.

 By Theorem \ref{lem2.5} and Proposition \ref{lem3.1}, $(\mathcal{PGF}(\xi)_n,\mathcal{PGF}(\xi)^{\bot_\xi}_n)$ is a complete hereditary cotorsion pair in $\mathrm{D}(R)$ with respect to $\xi$. It suffices to prove that $$\mathcal{PGF}(\xi)^{\bot_\xi}_n=\mathcal{P}(\xi)^{\bot_\xi}_n\cap\mathcal{PGF}(\xi)^{\bot_\xi}.$$ Clearly, $\mathcal{PGF}(\xi)^{\bot_\xi}_n\subseteq\mathcal{P}(\xi)^{\bot_\xi}_n\cap\mathcal{PGF}(\xi)^{\bot_\xi}$. Let $X\in\mathcal{P}(\xi)^{\bot_\xi}_n\cap\mathcal{PGF}(\xi)^{\bot_\xi}$ and $X\stackrel{\simeq}\rightarrow I$ be  a semi-injective resolution. Then $\mathrm{Z}^s(I)\in\textsf{P}_n^\perp\cap\textsf{PGF}^\bot$ by Lemma \ref{lem2.2} and Fact \ref{lem:2.2}(4) for all $s\in\mathbb{Z}$, it follows by Theorem \ref{lem2.5} that $X\in\mathcal{PGF}(\xi)^{\bot_\xi}_n$.
\end{proof}

\bigskip \centerline {\bf Acknowledgements}
\bigskip
This research was partially supported by National Natural Science Foundation of China (12571035).

\end{document}